\documentclass[british,11pt]{article}

\usepackage[T1]{fontenc}
\usepackage[utf8]{inputenc}
\usepackage{geometry}
\usepackage{color}
\usepackage{babel}
\usepackage{amsmath}
\usepackage{amsthm}
\usepackage{amssymb}
\usepackage{graphicx}
\usepackage[unicode=true,pdfusetitle,
 bookmarks=true,bookmarksnumbered=false,bookmarksopen=false,
 breaklinks=false,pdfborder={0 0 1},backref=false,colorlinks=true]
 {hyperref}

\makeatletter

\newcommand{\lyxdot}{.}

\newcommand{\lyxaddress}[1]{
	\par {\raggedright #1
	\vspace{1.4em}
	\noindent\par}
}
\theoremstyle{plain}
\newtheorem{thm}{\protect\theoremname}
\theoremstyle{definition}
\newtheorem{defn}[thm]{\protect\definitionname}
\theoremstyle{remark}
\newtheorem{rem}[thm]{\protect\remarkname}
\theoremstyle{plain}
\newtheorem{cor}[thm]{\protect\corollaryname}
\ifx\proof\undefined
\newenvironment{proof}[1][\protect\proofname]{\par
	\normalfont\topsep6\p@\@plus6\p@\relax
	\trivlist
	\itemindent\parindent
	\item[\hskip\labelsep\scshape #1]\ignorespaces
}{%
	\endtrivlist\@endpefalse
}
\providecommand{\proofname}{Proof}
\fi

\makeatother

\providecommand{\corollaryname}{Corollary}
\providecommand{\definitionname}{Definition}
\providecommand{\remarkname}{Remark}
\providecommand{\theoremname}{Theorem}

\begin{document}
\title{Non-resonant invariant foliations of quasi-periodically forced systems}
\author{Robert Szalai}
\date{15th March 2024}
\maketitle

\lyxaddress{School of Engineering Mathematics and Technology, University of Bristol,
Ada Lovelace Building, Tankard's Close, Bristol BS8 1TW, email: r.szalai@bristol.ac.uk}
\begin{abstract}
We show the existence and uniqueness of invariant foliations about
invariant tori in analytic discrete-time dynamical systems. The parametrisation
method is used prove the result. Our theory is a foundational block
of data-driven model order reduction, that can only be carried out
using invariant foliations. The theory is illustrated by two mechanical
examples, where instantaneous frequencies and damping ratios are calculated
about the invariant tori.
\end{abstract}



\section{Introduction}

It was found after an exhaustive review of possible ROM representations
that invariant foliations are the only mathematical structure that
can be fitted to existing data \cite{Szalai2023Fol}. The theory of
invariant foliations is intertwined with that of invariant manifolds
\cite{PalisDeMelo1982}. Historically, stable, unstable and centre
manifolds and otherwise normally-hyperbolic invariant manifolds were
at the focus of investigation \cite{hirsch1970,Fenichel,Mane1978,wiggins2013normally}.
Later, it was found that hyperbolicity can also be achieved in a nonlinear
coordinate systems about an invariant object, which led to non-resonant
invariant manifolds \cite{delaLlave1997}. The parametrisation method
was introduced to provide a comprehensive theoretical framework and
prove existence and uniqueness of non-resonant invariant manifolds
under general circumstances, for example in infinite dimensions \cite{CabreLlave2003,CabreP3-2005}.
Moreover, the parametrisation method proved to be successful in inspiring
numerical techniques \cite{Haro2016,VIZZACCARO2021normalForm,Haller2016}
to calculate invariant manifolds. The parametrisation method is similar
to normal form theory \cite{wiggins2003introduction}, which also
considers non-resonant terms.

Invariant foliations received less attention and the focus was mainly
on stable and unstable foliations near stable and unstable manifolds,
sometimes bundled together with the centre manifold \cite{BatesFoliations2000,AulbachFoliation2003}.
Non-resonant invariant foliations about periodic points were shown
to exist and unique in \cite{Szalai2020ISF}. Invariant foliations
may exist under less restrictive circumstances. Given their potential
to obtain data-driven reduced order models (ROM), exploring such situations
is warranted. In this paper we explore the case of quasi-periodic
systems and the existence of invariant foliations about invariant
tori. The theory follows a similar approach to \cite{HaroTori2006},
in that we use the parametrisation method to describe the geometry
of the foliation and the Kantorovich-Newton theorem to prove its existence
and uniqueness.

Koopman eigenfunctions that correspond to point spectra also describe
foliations \cite{Mauroy201319}. The Koopman operator is globally
defined, hence does not require the existence of a periodic or quasi-periodic
orbit. The drawback of the Koopman approach is that the dynamics on
the invariant foliation is assumed to be linear. Whenever the dynamics
is not linearisable, continuous spectrum occurs \cite{Mezic2005}.
The numerical treatment of the continuous spectrum requires careful
attention \cite{Colbrook2023}, and it is not clear that accounting
for a large number of spectral points still leads to a ROM. We also
note that being linearisable is a generic property, and as we can
see from our examples, some nonlinear phenomena can be recovered from
seemingly linear systems if we consider the nonlinear frame they are
defined in.

We also note that uniqueness results are local and numerically fragile.
Under most circumstances, there is a family of invariant manifolds/foliations
that satisfy the invariance equations, so we need a constraint that
singles one out. The parametrisation method \cite{CabreLlave2003}
uses smoothness as a deciding factor, while exponential rates can
also be used \cite{Llave1995}, at least for pseudo-stable manifolds.
As was pointed out in \cite{delaLlave1997}, the two definitions may
not lead to the same result. From the perspective of applications,
smoothness of a manifold is very difficult to establish, especially
when data is involved.

While the present paper is restricted to the neighbourhood of quasi-periodic
tori, the direction of travel for our investigation is to extend the
theory of invariant foliations to a large class of systems, preferably
with multiple (complex) attractors.

The structure of the paper is the following. We first describe the
set-up for our analysis, which includes a quick review of exponential
dichotomies for non-autonomous systems. The second section states
and proves our theorem about the existence and uniqueness of invariant
foliations for quasi-periodically forced systems. The third section
details the numerical methods to calculate the foliation, which is
necessary, given that there is significant difference between the
theoretical results and the discretised solution of the invariance
equation. Finally, in the fourth section, we illustrate our results
through two examples. The appendix includes known results about the
existence of invariant tori and we recall the Kantorovich-Newton theorem
for completeness.

\subsection{\label{subsec:Setup}Problem setting}

We assume a deterministic and discrete-time system in the form of
\begin{equation}
\begin{array}{rl}
\boldsymbol{x}_{k+1} & \negthickspace\negthickspace\negthickspace=\boldsymbol{F}\left(\boldsymbol{x}_{k},\boldsymbol{\theta}_{k}\right)\\
\boldsymbol{\theta}_{k+1} & \negthickspace\negthickspace\negthickspace=\boldsymbol{\theta}_{k}+\boldsymbol{\omega}
\end{array},\;k=1,2,\ldots,\label{eq:MAPSysDef}
\end{equation}
where $\boldsymbol{x}_{k}\in X$, $\boldsymbol{\theta}_{k}\in\mathbb{T}^{d}$,
$\boldsymbol{F}:\mathbb{T}^{d}\times X\to X$ is an analytic function,
$X$ is an $n$-dimensional inner product vector space, $\mathbb{T}^{d}$
is the $d$-dimensional torus and $\boldsymbol{\omega}\in\mathbb{T}^{d}$
is a constant angle of rotation on the torus. We also assume an invariant
torus 
\begin{equation}
\mathcal{T}=\left\{ \boldsymbol{K}\left(\boldsymbol{\theta}\right):\boldsymbol{\theta}\in\mathbb{T}^{d}\right\} ,\label{eq:TOR-geometry}
\end{equation}
where $\boldsymbol{K}:\mathbb{T}^{d}\to X$ is an analytic function
and satisfies the invariance equation
\begin{equation}
\boldsymbol{K}\left(\boldsymbol{\theta}+\omega\right)=\boldsymbol{F}\left(\boldsymbol{K}\left(\boldsymbol{\theta}\right),\boldsymbol{\theta}\right).\label{eq:TOR-invariances}
\end{equation}

The linear dynamics about the invariant torus $\mathcal{T}$ holds
a key to finding invariant foliations. Therefore we define 
\[
\boldsymbol{A}\left(\boldsymbol{\theta}\right)=D_{1}\boldsymbol{F}\left(\boldsymbol{K}\left(\boldsymbol{\theta}\right),\boldsymbol{\theta}\right)
\]
and consider the system
\begin{equation}
\begin{array}{rl}
\boldsymbol{x}_{k+1} & \negthickspace\negthickspace\negthickspace=\boldsymbol{A}\left(\boldsymbol{\theta}_{k}\right)\boldsymbol{x}_{k}\\
\boldsymbol{\theta}_{k+1} & \negthickspace\negthickspace\negthickspace=\boldsymbol{\theta}_{k}+\boldsymbol{\omega}
\end{array},\;k=1,2,\ldots.\label{eq:ED-linsys}
\end{equation}
The phase space of (\ref{eq:ED-linsys}), denoted by $E=X\times\mathbb{T}^{d}$
is called the trivial vector bundle. The decomposition of $E$ is
a Whitney sum of sub-bundles, $E=E_{1}\oplus\cdots\oplus E_{m}$.
Each sub-bundle is composed of fibres, which means that for each $\boldsymbol{\theta}\in\mathbb{T}^{d}$,
$E_{j\boldsymbol{\theta}}$\textbf{ }is a subspace of $X$, and therefore
$X=E_{1\boldsymbol{\theta}}\oplus\cdots\oplus E_{m\boldsymbol{\theta}}$.
Note that the same decomposition can be defined for the adjoint trivial
bundle $E^{\star}=X^{\star}\times\mathbb{T}^{d}$, as well. For the
sub-bundles to be meaningful they need to be invariant under equation
(\ref{eq:ED-linsys}), which means that for each fibre we must have
\begin{equation}
E_{j\boldsymbol{\theta}+\boldsymbol{\omega}}=\boldsymbol{A}\left(\boldsymbol{\theta}\right)E_{j\boldsymbol{\theta}}\label{eq:ED-inv-bundle}
\end{equation}
or for the adjoint sub-bundle, 
\begin{equation}
E_{j\boldsymbol{\theta}}^{\star}=\boldsymbol{A}^{\star}\left(\boldsymbol{\theta}\right)E_{j\boldsymbol{\theta}+\boldsymbol{\omega}}^{\star}.\label{eq:ED-inv-adj-bundle}
\end{equation}

In order to make the calculations of invariant vector bundles computationally
accessible, we represent them by families of orthogonal matrices.
The two invariance equations (\ref{eq:ED-inv-bundle}), (\ref{eq:ED-inv-adj-bundle})
can be condensed into one, if we use a projection $\boldsymbol{P}:\mathbb{T}^{d}\to L\left(X,X\right)$,
in which case $E_{j\boldsymbol{\theta}}=\mathrm{rng}\:\boldsymbol{P}\left(\boldsymbol{\theta}\right)$
and $E_{j\boldsymbol{\theta}}^{\star}=\mathrm{rng}\:\boldsymbol{P}_{j}^{\star}\left(\boldsymbol{\theta}\right)=\left(\ker\boldsymbol{P}_{j}\left(\boldsymbol{\theta}\right)\right)^{\perp}$.
The invariance equation for the projection is 
\begin{equation}
\boldsymbol{P}\left(\theta+\omega\right)\boldsymbol{A}\left(\theta\right)=\boldsymbol{A}\left(\theta\right)\boldsymbol{P}\left(\theta\right).\label{eq:ED-invariance}
\end{equation}
A projection $\boldsymbol{P}$ can be factored into a product $\boldsymbol{P}\left(\boldsymbol{\theta}\right)=\boldsymbol{W}\left(\boldsymbol{\theta}\right)\boldsymbol{U}\left(\boldsymbol{\theta}\right)$,
where $\boldsymbol{W}$ is an orthogonal operator having the same
range as $\boldsymbol{P}$ and therefore $\boldsymbol{U}\left(\boldsymbol{\theta}\right)=\boldsymbol{W}^{T}\left(\boldsymbol{\theta}\right)\boldsymbol{P}\left(\boldsymbol{\theta}\right)$.
Due to $\boldsymbol{P}$ being a projection, we also have $\boldsymbol{U}\left(\boldsymbol{\theta}\right)\boldsymbol{W}\left(\boldsymbol{\theta}\right)=\boldsymbol{I}$.
Substituting the factorisation of $\boldsymbol{P}$ into invariance
equation (\ref{eq:ED-invariance}) produces
\begin{equation}
\boldsymbol{W}\left(\boldsymbol{\theta}+\omega\right)\boldsymbol{U}\left(\boldsymbol{\theta}+\omega\right)\boldsymbol{A}\left(\boldsymbol{\theta}\right)=\boldsymbol{A}\left(\boldsymbol{\theta}\right)\boldsymbol{W}\left(\boldsymbol{\theta}\right)\boldsymbol{U}\left(\boldsymbol{\theta}\right).\label{eq:ED-expanded-invariance}
\end{equation}
Multiplying (\ref{eq:ED-expanded-invariance}) by $\boldsymbol{U}\left(\boldsymbol{\theta}+\boldsymbol{\omega}\right)$
from the left yields 
\begin{equation}
\boldsymbol{U}\left(\boldsymbol{\theta}+\boldsymbol{\omega}\right)\boldsymbol{A}\left(\boldsymbol{\theta}\right)=\boldsymbol{\Lambda}\left(\boldsymbol{\theta}\right)\boldsymbol{U}\left(\boldsymbol{\theta}\right),\label{eq:ED-left-invariance}
\end{equation}
where $\boldsymbol{\Lambda}_{i}\left(\boldsymbol{\theta}\right)=\boldsymbol{U}\left(\boldsymbol{\theta}+\boldsymbol{\omega}\right)\boldsymbol{A}\left(\boldsymbol{\theta}\right)\boldsymbol{W}\left(\boldsymbol{\theta}\right)$
and multiplying (\ref{eq:ED-expanded-invariance}) by $\boldsymbol{W}\left(\boldsymbol{\theta}\right)$
from the right yields 
\begin{equation}
\boldsymbol{W}\left(\boldsymbol{\theta}+\boldsymbol{\omega}\right)\boldsymbol{\Lambda}\left(\boldsymbol{\theta}\right)=\boldsymbol{A}\left(\boldsymbol{\theta}\right)\boldsymbol{W}\left(\boldsymbol{\theta}\right).\label{eq:ED-right-invariance}
\end{equation}
Equations (\ref{eq:ED-left-invariance}) and (\ref{eq:ED-right-invariance})
are precursors to the invariance equations of foliations and manifolds.
One can also think of equations (\ref{eq:ED-left-invariance}) and
(\ref{eq:ED-right-invariance}) as generalisations for eigenvalues
and eigenvectors, where $\boldsymbol{\Lambda}$ takes the place of
the eigenvalue, $\boldsymbol{W}$ the right eigenvector and $\boldsymbol{U}$
the left eigenvector.

Invariance alone is not sufficient to characterise the linear dynamics,
we also need a dynamic property that decides how to split up the linear
system (\ref{eq:ED-linsys}). This notion is the exponential dichotomy,
which has been introduced in \cite{SACKER1978320} for non-autonomous
systems, though similar notions were used before.
\begin{defn}
Given a projection $\boldsymbol{P}$, satisfying the invariance equation
(\ref{eq:ED-invariance}), we say that there is an exponential dichotomy
for a positive real number $\alpha>0$ if there exists $C>0$ such
that for all $k\ge0$ and all $\boldsymbol{\theta}\in\mathbb{T}^{d}$
the inequalities hold
\begin{align*}
\left|\boldsymbol{A}\left(\boldsymbol{\theta}+\left(k-1\right)\boldsymbol{\omega}\right)\cdots\boldsymbol{A}\left(\boldsymbol{\theta}+\boldsymbol{\omega}\right)\boldsymbol{A}\left(\boldsymbol{\theta}\right)\boldsymbol{P}\left(\boldsymbol{\theta}\right)\right| & \le C\alpha^{k},\\
\left|\boldsymbol{A}^{-1}\left(\boldsymbol{\theta}-k\boldsymbol{\omega}\right)\cdots\boldsymbol{A}^{-1}\left(\boldsymbol{\theta}-\boldsymbol{\omega}\right)\left(\boldsymbol{I}-\boldsymbol{P}\left(\boldsymbol{\theta}\right)\right)\right| & \le C\alpha^{-k}.
\end{align*}
The set of $\alpha$ values for which there exists an exponential
dichotomy is called the resolvent set and denoted by $\mathcal{R}\left(\boldsymbol{A};\boldsymbol{\omega}\right)$.
The complement $\varSigma\left(\boldsymbol{A};\boldsymbol{\omega}\right)=\mathbb{R}^{+}\setminus\mathcal{R}\left(\boldsymbol{A};\boldsymbol{\omega}\right)$
is called the dichotomy spectrum of system (\ref{eq:ED-linsys}).
\end{defn}
\begin{rem}
Matrix $\boldsymbol{A}\left(\boldsymbol{\theta}-\boldsymbol{\omega}\right)$
only needs to be invertible on the range of $\boldsymbol{I}-\boldsymbol{P}\left(\boldsymbol{\theta}\right)$.
This allows us to have $\boldsymbol{A}\left(\boldsymbol{\theta}\right)\boldsymbol{P}\left(\boldsymbol{\theta}\right)=\boldsymbol{0}$
and arbitrarily small $\alpha>0$.
\end{rem}
\begin{cor}
\label{cor:ED-linear-solution}Equation
\[
\boldsymbol{x}\left(\boldsymbol{\theta}+\boldsymbol{\omega}\right)=\boldsymbol{A}\left(\boldsymbol{\theta}\right)\boldsymbol{x}\left(\boldsymbol{\theta}\right)+\boldsymbol{\eta}\left(\boldsymbol{\theta}\right),
\]
where $\boldsymbol{\eta}$ and $\boldsymbol{A}$ are analytic has
a unique solution within the set of analytic functions if $1\in\mathcal{R}\left(\boldsymbol{A};\boldsymbol{\omega}\right)$.
\end{cor}
\begin{proof}
We use a contraction mapping argument. The space $X\times\mathbb{T}^{d}$
is separated into two by the projection $\boldsymbol{P}$. On the
range of $\boldsymbol{P}$ we use the iteration $\boldsymbol{x}^{\left(k+1\right)}\left(\boldsymbol{\theta}+\boldsymbol{\omega}\right)=\boldsymbol{A}\left(\boldsymbol{\theta}\right)\boldsymbol{x}\left(k\right)\left(\boldsymbol{\theta}\right)+\boldsymbol{P}\left(\boldsymbol{\theta}+\boldsymbol{\omega}\right)\boldsymbol{\eta}\left(\boldsymbol{\theta}\right)$,
while on the range of $\boldsymbol{I}-\boldsymbol{P}$, we use the
iteration 
\[
\boldsymbol{x}^{\left(k+1\right)}\left(\boldsymbol{\theta}\right)=\boldsymbol{A}^{-1}\left(\boldsymbol{\theta}\right)\boldsymbol{x}^{\left(k\right)}\left(\boldsymbol{\theta}+\boldsymbol{\omega}\right)-\left(\boldsymbol{I}-\boldsymbol{P}\left(\boldsymbol{\theta}\right)\right)\boldsymbol{A}^{-1}\left(\boldsymbol{\theta}\right)\boldsymbol{\eta}\left(\boldsymbol{\theta}\right)
\]
to find a convergent solution.
\end{proof}
There are a few properties of the resolvent set or the spectrum that
we need to use. These properties are widely described in the literature,
in particular \cite{Aulbach-II,SACKER1978320}. The resolvent set
consists of a finite number of open intervals, which we denote by
\begin{gather}
\mathcal{R}_{0}=\left(0,\alpha_{1}\right),\mathcal{R}_{1}=\left(\beta_{1},\alpha_{2}\right),\ldots,\mathcal{R}_{m}=\left(\beta_{m},\infty\right)\;\text{or}\label{eq:ED-dichotomy-A}\\
\mathcal{R}_{1}=\left(\beta_{1},\alpha_{2}\right),\mathcal{R}_{2}=\left(\beta_{2},\alpha_{3}\right),\ldots,\mathcal{R}_{m}=\left(\beta_{m},\infty\right)\;\text{if\;\ensuremath{\alpha_{1}=0}}\nonumber 
\end{gather}
with ordering $\beta_{i}<\alpha_{i+1}$ and $\alpha_{i}<\beta_{i}$,
such that $\mathcal{R}\left(\boldsymbol{A};\boldsymbol{\omega}\right)=\bigcup_{j=0\,\text{or}\,1}^{m}\mathcal{R}_{j}$.
Note that the maximum $m$ is $\dim X$. Within a resolvent interval
for $\alpha\in\mathcal{R}_{j}$, the invariant projection $\boldsymbol{P}$
is independent of $\alpha$ and therefore can be denoted by $\boldsymbol{P}_{i}^{ED}$,
where $1\le i\le m\le n$. We can also create spectral intervals 

\begin{align}
\Sigma_{1} & =\left[\alpha_{1},\beta_{1}\right],\;\Sigma_{2}=\left[\alpha_{2},\beta_{2}\right],\ldots,\;\Sigma_{_{m}}=\left[\alpha_{m},\beta_{m}\right],\;\text{or}\label{eq:ED-spectrum-A}\\
\Sigma_{1} & =\left(\alpha_{1},\beta_{1}\right],\;\Sigma_{2}=\left[\alpha_{2},\beta_{2}\right],\ldots,\;\Sigma_{m}=\left[\alpha_{m},\beta_{m}\right]\;\text{if}\;\alpha_{1}=0,\label{eq:ED-spectrum-B}
\end{align}
that are sometimes just points. The full spectrum then becomes $\varSigma\left(\boldsymbol{A};\boldsymbol{\omega}\right)=\bigcup_{j=1}^{m}\Sigma_{j}$.
Note that if $\boldsymbol{P}_{i}^{ED}$ is an invariant projection,
so is $\boldsymbol{I}-\boldsymbol{P}_{i}^{ED}$. We also have the
inclusions 
\begin{equation}
\mathrm{rng}\boldsymbol{P}_{1}^{ED}\subset\mathrm{rng}\boldsymbol{P}_{2}^{ED}\subset\cdots\subset\mathrm{rng}\boldsymbol{P}_{m}^{ED}=X,\label{eq:ED-inclusions}
\end{equation}
and therefore $\boldsymbol{P}_{m}^{ED}=\boldsymbol{I}$. The inclusions
can be used to decompose the system (\ref{eq:ED-linsys}) into the
same number of subsystems as there are spectral intervals. We can
then define the spectral projections by
\[
\boldsymbol{P}_{1}=\boldsymbol{P}_{1}^{ED},\boldsymbol{P}_{2}=\boldsymbol{P}_{2}^{ED}-\boldsymbol{P}_{1}^{ED},\ldots,\boldsymbol{P}_{m}=\boldsymbol{P}_{m}^{ED}-\boldsymbol{P}_{m-1}^{ED},
\]
which are similarly invariant, that is, $\boldsymbol{P}_{i}\left(\boldsymbol{\theta}+\boldsymbol{\omega}\right)\boldsymbol{A}\left(\boldsymbol{\theta}\right)=\boldsymbol{A}\left(\boldsymbol{\theta}\right)\boldsymbol{P}_{i}\left(\boldsymbol{\theta}\right)$.
The inclusion property (\ref{eq:ED-inclusions}) implies that 
\[
\mathrm{rng}\boldsymbol{P}_{1}\oplus\mathrm{rng}\boldsymbol{P}_{2}\oplus\cdots\oplus\mathrm{rng}\boldsymbol{P}_{m}=X.
\]

The low-dimensional matrices then have the exponential bounds
\begin{align}
\left|\boldsymbol{\Lambda}_{i}\left(\boldsymbol{\theta}+\left(k-1\right)\boldsymbol{\omega}\right)\cdots\boldsymbol{\Lambda}_{i}\left(\boldsymbol{\theta}+\boldsymbol{\omega}\right)\boldsymbol{\Lambda}_{i}\left(\boldsymbol{\theta}\right)\right| & \le C\beta^{k}, & \forall\beta>\beta_{i},\label{eq:ED-straight-bound}\\
\left|\boldsymbol{\Lambda}_{i}\left(\boldsymbol{\theta}+\left(k-1\right)\boldsymbol{\omega}\right)\cdots\boldsymbol{\Lambda}_{i}\left(\boldsymbol{\theta}+\boldsymbol{\omega}\right)\boldsymbol{\Lambda}_{i}\left(\boldsymbol{\theta}\right)\right| & \ge C^{-1}\alpha^{k}, & \forall\alpha<\alpha_{i},\label{eq:ED-straight-lower}\\
\left|\boldsymbol{\Lambda}_{i}^{-1}\left(\boldsymbol{\theta}-k\boldsymbol{\omega}\right)\cdots\boldsymbol{\Lambda}_{i}^{-1}\left(\boldsymbol{\theta}-2\boldsymbol{\omega}\right)\boldsymbol{\Lambda}_{i}^{-1}\left(\boldsymbol{\theta}-\omega\right)\right| & \le C\alpha^{-k}, & \forall\alpha<\alpha_{i} & \;\text{if}\;\alpha_{i}\neq0\label{eq:ED-inverse-bound}\\
\left|\boldsymbol{\Lambda}_{i}^{-1}\left(\boldsymbol{\theta}-k\boldsymbol{\omega}\right)\cdots\boldsymbol{\Lambda}_{i}^{-1}\left(\boldsymbol{\theta}-2\boldsymbol{\omega}\right)\boldsymbol{\Lambda}_{i}^{-1}\left(\boldsymbol{\theta}-\omega\right)\right| & \ge C^{-1}\beta^{-k}, & \forall\beta>\beta_{i}.\label{eq:ED-inverse-lower}
\end{align}

\begin{cor}
\label{cor:semi-diagonalise}The linear system (\ref{eq:ED-linsys})
can be block-diagonalised into 
\begin{equation}
\begin{array}{rl}
\boldsymbol{x}_{k+1} & \negthickspace\negthickspace\negthickspace=\tilde{\boldsymbol{A}}\left(\boldsymbol{\theta}_{k}\right)\boldsymbol{x}_{k}\\
\boldsymbol{\theta}_{k+1} & \negthickspace\negthickspace\negthickspace=\boldsymbol{\theta}_{k}+\boldsymbol{\omega}
\end{array},\;k=1,2,\ldots,\label{eq:ED-Jordan}
\end{equation}
where
\[
\tilde{\boldsymbol{A}}\left(\boldsymbol{\theta}\right)=\begin{pmatrix}\boldsymbol{\Lambda}_{1}\left(\boldsymbol{\theta}\right) &  & \mathbf{0}\\
 & \ddots\\
\mathbf{0} &  & \boldsymbol{\Lambda}_{m}\left(\boldsymbol{\theta}\right)
\end{pmatrix}=\begin{pmatrix}\boldsymbol{U}_{1}\left(\boldsymbol{\theta}+\boldsymbol{\omega}\right)\\
\vdots\\
\boldsymbol{U}_{m}\left(\boldsymbol{\theta}+\boldsymbol{\omega}\right)
\end{pmatrix}\boldsymbol{A}\left(\boldsymbol{\theta}\right)\begin{pmatrix}\boldsymbol{W}_{1}\left(\boldsymbol{\theta}\right) & \cdots & \boldsymbol{W}_{m}\left(\boldsymbol{\theta}\right)\end{pmatrix}.
\]
\end{cor}
\begin{defn}
Equation (\ref{eq:ED-linsys}) is called \emph{reducible}, if $\boldsymbol{\Lambda}_{1},\ldots,\boldsymbol{\Lambda}_{m}$
can be made constant.
\end{defn}
As mentioned in section (\ref{subsec:Setup}), it is worth repeating
that if we are dealing with a system with a single forcing frequency
in continuous time, reducibility is guaranteed by Floquet theory.
Not all systems are reducible, and this can be the source of strange
non-chaotic attractors \cite{GrebogiSNA}. However, from a numerical
perspective, we can always compute the eigenvalues of the discretised
transfer operator and therefore it is difficult to tell if a system
is in fact irreducible.

\section{Invariant foliation}

Within the context of this paper, an invariant foliation is a family
of differentiable manifolds, called leaves, that fill up the phase
space $X\times\mathbb{T}^{d}$ and are mapped onto each other by the
map (\ref{eq:MAPSysDef}). To represent a foliation, we use an analytic
function $\boldsymbol{U}:X\times\mathbb{T}^{d}\to Z$, where $Z$
is an inner product vector space such that $\dim Z<\dim X$. The leaves
of the foliation are parametrised by the elements of $Z\times\mathbb{T}^{d}$
and given as level surfaces of the representing function $\boldsymbol{U}$
in the form of
\[
\mathcal{L}_{\boldsymbol{z},\boldsymbol{\theta}}=\left\{ \left(\boldsymbol{x},\boldsymbol{\theta}\right)\in X\times\mathbb{T}^{d}:\boldsymbol{U}\left(\boldsymbol{x},\boldsymbol{\theta}\right)=\boldsymbol{z}\right\} .
\]
In our notation the foliation is the collection
\begin{equation}
\mathcal{F}=\left\{ \mathcal{L}_{\boldsymbol{z},\boldsymbol{\theta}}:\left(\boldsymbol{z},\boldsymbol{\theta}\right)\in Z\times\mathbb{T}^{d}\right\} .\label{eq:FOIL-geometry}
\end{equation}
Foliation $\mathcal{F}$ is invariant if there exists a map $\boldsymbol{R}:Z\times\mathbb{T}^{d}\to Z$,
such that the invariance equation 
\begin{equation}
\boldsymbol{R}\left(\boldsymbol{U}\left(\boldsymbol{x},\boldsymbol{\theta}\right),\boldsymbol{\theta}\right)=\boldsymbol{U}\left(\boldsymbol{F}\left(\boldsymbol{x},\boldsymbol{\theta}\right),\boldsymbol{\theta}+\boldsymbol{\omega}\right)\label{eq:FOIL-invariance}
\end{equation}
holds. Equation (\ref{eq:FOIL-invariance}) is the same as the inclusion
$\boldsymbol{F}\left(\mathcal{L}_{\boldsymbol{z},\boldsymbol{\theta}},\boldsymbol{\theta}\right)\subset\mathcal{L}_{\boldsymbol{R}\left(\boldsymbol{z},\boldsymbol{\theta}\right),\boldsymbol{\theta}+\boldsymbol{\omega}}$.
A schematic of a foliation about an invariant closed curve can be
seen in figure \ref{fig:FOIL-schematic}.
\begin{figure}
\begin{centering}
\includegraphics[width=0.4\textwidth]{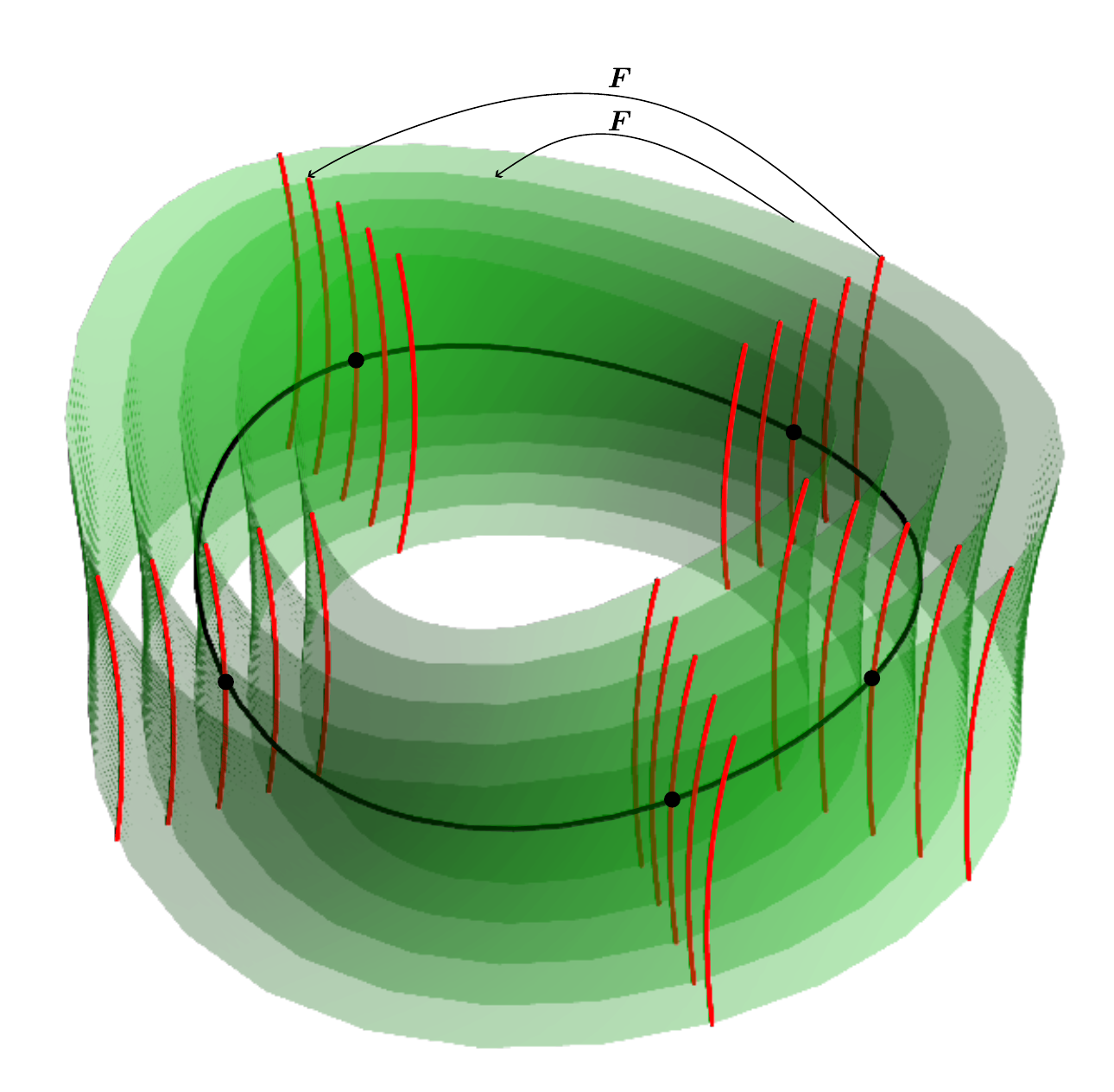}
\par\end{centering}
\caption{\label{fig:FOIL-schematic}Invariant foliation about an invariant
closed curve (torus), which is shown in black. The invariance equation
maps each red curve into another red curve, which all lie on a green
surface as parametrised by the $\boldsymbol{\theta}$ variable.}
\end{figure}

Without losing generality, we assume that $\boldsymbol{K}=\boldsymbol{0}$,
which can be achieved by defining the $\tilde{\boldsymbol{F}}\left(\boldsymbol{x},\boldsymbol{\theta}\right)=\boldsymbol{F}\left(\boldsymbol{K}\left(\boldsymbol{\theta}\right)+\boldsymbol{x},\boldsymbol{\theta}\right)-\boldsymbol{K}\left(\boldsymbol{\theta}+\boldsymbol{\omega}\right)$
and dropping the tilde from $\tilde{\boldsymbol{F}}$. To normalise
the relation between $\boldsymbol{R}$ and $\boldsymbol{U}$, we assume
that $\boldsymbol{U}\left(\boldsymbol{K}\left(\boldsymbol{\theta}\right),\boldsymbol{\theta}\right)=\boldsymbol{0}$
and therefore $\boldsymbol{R}\left(\boldsymbol{0},\boldsymbol{\theta}\right)=\boldsymbol{0}$.
An invariant foliation is defined for a set of invariant vector bundles
about the torus $\mathcal{T}$. The following theorem lays down the
conditions when such foliation exists and unique.
\begin{thm}
\label{thm:exists-unique}Assume an invariant torus (\ref{eq:TOR-geometry})
and a linearised system about the torus in the form of (\ref{eq:ED-linsys}).
Also assume that the linear system has dichotomy spectral intervals
given by (\ref{eq:ED-spectrum-A}) or (\ref{eq:ED-spectrum-B}) and
that $\beta_{m}<1$. Pick one or more spectral intervals using a non-empty
index set 
\[
\mathcal{I}\subset\left\{ 1,2,\ldots,m\right\} 
\]
with the condition that $\alpha_{j}\neq0$ for all $j\in\mathcal{I}$
so that linear the system (\ref{eq:ED-linsys}) restricted to the
subset of the spectrum 
\[
\Sigma_{\mathcal{F}}=\bigcup_{j\in\mathcal{I}}\Sigma_{j}
\]
is invertible. Define the spectral quotient as 
\[
\beth_{\mathcal{I}}=\frac{\min_{j\in\mathcal{I}}\log\alpha_{j}}{\log\beta_{m}}.
\]
If the non-resonance conditions 
\begin{equation}
1\notin\left[\beta_{i_{0}}^{-1}\alpha_{i_{1}}\cdots\alpha_{i_{j}},\alpha_{i_{0}}^{-1}\beta_{i_{1}}\cdots\beta_{i_{j}}\right]\label{eq:FOIL-non-resonance-1}
\end{equation}
hold for $i_{0}\in\mathcal{I}$, $i_{1},\ldots,i_{j}\in\left\{ 1,2,\ldots,m\right\} $
such that there exists $i_{k}\notin\mathcal{I}$ and $2\le j<\beth_{\mathcal{I}}+1$
then
\begin{enumerate}
\item there exists a unique invariant foliation defined by analytic functions
$\boldsymbol{U}$ and $\boldsymbol{R}$ satisfying equation (\ref{eq:FOIL-invariance})
in a sufficiently small neighbourhood of the torus $\mathcal{T}$
such that $D_{1}\boldsymbol{U}\left(\boldsymbol{0},\boldsymbol{\theta}\right)=\bigoplus_{j\in\mathcal{I}}\boldsymbol{U}_{j}\left(\boldsymbol{\theta}\right)$;
\item the nonlinear map $\boldsymbol{R}$ is a polynomial, which in its
simplest form contains terms for which the internal non-resonance
conditions (\ref{eq:FOIL-non-resonance-1}) with $i_{0},i_{1},\ldots,i_{j}\in\mathcal{I}$
and $j<\beth_{\mathcal{I}}+1$ does not hold.
\end{enumerate}
\end{thm}
\begin{proof}
The proof has three steps. First we find a polynomial approximation
of the invariant foliation which takes into account all possible resonance
conditions. Second, we calculate an exact and unique correction to
the polynomial expansion of the foliation. Finally, we establish that
despite our choice of terms in the polynomial expansion, the calculated
foliation is unique as a geometric object.

To make calculations simpler, we first shift the torus $\mathcal{M}$
to the origin, such that the new map is $\hat{\boldsymbol{F}}\left(\boldsymbol{x},\boldsymbol{\theta}\right)=\boldsymbol{F}\left(\boldsymbol{K}\left(\boldsymbol{\theta}\right)+\boldsymbol{x},\boldsymbol{\theta}\right)$
and drop the hat from $\hat{\boldsymbol{F}}$. This normalises our
system such that $\boldsymbol{F}\left(\boldsymbol{0},\boldsymbol{\theta}\right)=\boldsymbol{0}$.
In addition, we choose $\boldsymbol{U}\left(\boldsymbol{0},\boldsymbol{\theta}\right)=\boldsymbol{0}$,
and then it follows from (\ref{eq:FOIL-invariance}) that $\boldsymbol{R}\left(\boldsymbol{0},\boldsymbol{\theta}\right)=\boldsymbol{0}$.

Continuing with our plan, we expand the two unknowns of the invariance
equation (\ref{eq:FOIL-invariance}) in the form of
\begin{align}
\boldsymbol{R}\left(\boldsymbol{z},\boldsymbol{\theta}\right) & =\sum_{j=1}^{\sigma}\boldsymbol{R}^{j}\left(\boldsymbol{\theta}\right)\boldsymbol{z}^{\otimes j},\label{eq:R-def-1-1}\\
\boldsymbol{U}\left(\boldsymbol{x},\boldsymbol{\theta}\right) & =\sum_{j=1}^{\sigma}\boldsymbol{U}^{j}\left(\boldsymbol{\theta}\right)\boldsymbol{x}^{\otimes j}+\boldsymbol{U}^{>}\left(\boldsymbol{x},\boldsymbol{\theta}\right),\label{eq:W-def-1-1}
\end{align}
where $\sigma$ is a polynomial order that will be clarified later.
In the first step of the solution of (\ref{eq:FOIL-invariance}),
we solve for the polynomial coefficients $\boldsymbol{R}^{j}$ and
$\boldsymbol{U}^{j}$. In the second step, we solve for the correction
term $\boldsymbol{U}^{>}$. In our case, the nonlinear map $\boldsymbol{R}$
will remain a polynomial. 

The linear terms $\boldsymbol{R}^{1}$, $\boldsymbol{U}^{1}$ are
recovered from the linearisation of the invariance equation (\ref{eq:FOIL-invariance})
about the origin, which is
\begin{align}
\boldsymbol{R}^{1}\left(\boldsymbol{\theta}\right)\boldsymbol{U}^{1}\left(\boldsymbol{\theta}\right) & =\boldsymbol{U}^{1}\left(\boldsymbol{\theta}+\boldsymbol{\omega}\right)\boldsymbol{A}\left(\boldsymbol{\theta}\right),\label{eq:FOIL-inv-linear}
\end{align}
where $\boldsymbol{A}\left(\boldsymbol{\theta}\right)=D_{1}\boldsymbol{F}\left(\boldsymbol{0},\boldsymbol{\theta}\right)$.
The matrices $\boldsymbol{R}^{1}$ and $\boldsymbol{U}^{1}$ can be
found from the dichotomy spectral decomposition of matrix $\boldsymbol{A}$
as per corollary \ref{cor:semi-diagonalise}. Noticing that the linearised
invariance equation (\ref{eq:FOIL-inv-linear}) is the same as (\ref{eq:ED-left-invariance}),
we find that
\[
\boldsymbol{R}^{1}\left(\boldsymbol{\theta}\right)=\begin{pmatrix}\boldsymbol{\Lambda}_{i_{1}}\left(\boldsymbol{\theta}\right) &  & \boldsymbol{0}\\
 & \ddots\\
\boldsymbol{0} &  & \boldsymbol{\Lambda}_{i_{\#\mathcal{I}}}\left(\boldsymbol{\theta}\right)
\end{pmatrix},\quad\boldsymbol{U}^{1}\left(\boldsymbol{\theta}\right)=\begin{pmatrix}\boldsymbol{U}_{i_{1}}\left(\boldsymbol{\theta}\right)\\
\vdots\\
\boldsymbol{U}_{i_{\#\mathcal{I}}}\left(\boldsymbol{\theta}\right)
\end{pmatrix},\quad\mathcal{\mathcal{I}}=\left\{ i_{1},i_{2},\ldots,i_{\#\mathcal{I}}\right\} .
\]
We solve for the nonlinear terms in increasing order, starting with
$j=2$. We can do this procedure, because terms of order $j$ only
depend on unknown terms of order $j$ or lower. The equation for the
order-$j$ terms is given by
\begin{equation}
\boldsymbol{R}^{1}\left(\boldsymbol{\theta}\right)\boldsymbol{U}^{j}\left(\boldsymbol{\theta}\right)\boldsymbol{x}^{\otimes j}+\boldsymbol{R}^{j}\left(\boldsymbol{\theta}\right)\left(\boldsymbol{U}^{1}\left(\boldsymbol{\theta}\right)\boldsymbol{x}\right)^{\otimes j}-\boldsymbol{U}^{j}\left(\boldsymbol{\theta}+\boldsymbol{\omega}\right)\left(\boldsymbol{A}\left(\boldsymbol{\theta}\right)\boldsymbol{x}\right)^{\otimes j}=\boldsymbol{\Gamma}^{j}\left(\boldsymbol{\theta}\right)\boldsymbol{x}^{\otimes j},\label{eq:FOIL-homological}
\end{equation}
where $\boldsymbol{\Gamma}^{j}$ is composed of polynomial coefficients
of $\boldsymbol{R}$ and $\boldsymbol{U}$ of order less than $j$,
and terms of $\boldsymbol{F}$. Equation (\ref{eq:FOIL-homological})
is linear for the polynomial coefficients $\boldsymbol{R}^{j}$ and
$\boldsymbol{U}^{j}$. Equation (\ref{eq:FOIL-homological}) can also
be split up into independent lower dimensional equations on invariant
subspaces, that are created from the invariant subspaces of $\boldsymbol{A}$.
For the vector space $Z$, we need to introduce the decomposition
such that $\boldsymbol{V}_{i}^{T}\boldsymbol{R}^{1}\left(\boldsymbol{\theta}\right)\boldsymbol{V}_{i}=\boldsymbol{\Lambda}_{i}\left(\boldsymbol{\theta}\right)$,
the identity is $\boldsymbol{I}=\begin{pmatrix}\boldsymbol{V}_{i_{1}} & \cdots & \boldsymbol{V}_{i_{\#\mathcal{I}}}\end{pmatrix}$
and $\boldsymbol{V}_{p}^{T}\boldsymbol{V}_{q}=\boldsymbol{I}\delta_{pq}$,
where $\delta_{pq}$ is the Kronecker delta. We also decompose $\boldsymbol{U}^{j}$
using $\boldsymbol{V}_{i}$ in its range and $\boldsymbol{U}_{i}$
in its domain. This further decomposition leads to the expressions
\begin{align*}
\boldsymbol{R}^{j}\left(\boldsymbol{\theta}\right)\boldsymbol{z}^{\otimes j} & =\sum_{i_{0}\cdots i_{j}}\boldsymbol{V}_{i_{0}}\boldsymbol{R}_{i_{0}i_{1}\cdots i_{j}}^{j}\left(\boldsymbol{\theta}\right)\left[\boldsymbol{V}_{i_{1}}^{T}\boldsymbol{z}\otimes\cdots\otimes\boldsymbol{V}_{i_{j}}^{T}\boldsymbol{z}\right],\\
\boldsymbol{U}^{j}\left(\boldsymbol{\theta}\right)\boldsymbol{x}^{\otimes j} & =\sum_{i_{0}\cdots i_{j}}\boldsymbol{V}_{i_{0}}\boldsymbol{U}_{i_{0}i_{1}\cdots i_{j}}^{j}\left(\boldsymbol{\theta}\right)\left[\boldsymbol{U}_{i_{1}}\left(\boldsymbol{\theta}\right)\boldsymbol{x}\otimes\cdots\otimes\boldsymbol{U}_{i_{j}}\left(\boldsymbol{\theta}\right)\boldsymbol{x}\right],
\end{align*}
where the indices are $i_{0}\in\mathcal{I}$ and $i_{1},i_{2},\ldots,i_{j}\in\left\{ 1,2,\ldots,m\right\} $
for $\boldsymbol{U}$ and $i_{1},i_{2},\ldots,i_{j}\in\mathcal{I}$
for $\boldsymbol{R}$. Substituting $\boldsymbol{x}^{\otimes j}=\boldsymbol{W}_{i_{1}}\left(\boldsymbol{\theta}\right)\otimes\cdots\otimes\boldsymbol{W}_{i_{j}}\left(\boldsymbol{\theta}\right)$
into (\ref{eq:FOIL-homological}) and multiplying from the left by
$\boldsymbol{V}_{i_{0}}^{T}$ and finally noting that $\boldsymbol{U}_{p}\left(\boldsymbol{\theta}\right)\boldsymbol{W}_{q}\left(\boldsymbol{\theta}\right)=\boldsymbol{I}\delta_{pq}$
we find that 
\begin{align}
\boldsymbol{\Lambda}_{i_{0}}\left(\boldsymbol{\theta}\right)\boldsymbol{U}_{i_{0}i_{1}\cdots i_{j}}^{j}\left(\boldsymbol{\theta}\right)+\boldsymbol{R}_{i_{0}i_{1}\cdots i_{j}}^{j}\left(\boldsymbol{\theta}\right)-\boldsymbol{U}_{i_{0}i_{1}\cdots i_{j}}^{j}\left(\boldsymbol{\theta}+\boldsymbol{\omega}\right)\left[\boldsymbol{\Lambda}_{i_{1}}\left(\boldsymbol{\theta}\right)\otimes\cdots\otimes\boldsymbol{\Lambda}_{i_{j}}\left(\boldsymbol{\theta}\right)\right] & =\boldsymbol{\Gamma}_{i_{0}i_{1}\cdots i_{j}}^{j}\left(\boldsymbol{\theta}\right)\;\text{for}\;i_{0},i_{1},\ldots,i_{j}\in\mathcal{I},\label{eq:FOIL-poly-decomp-int}\\
\boldsymbol{\Lambda}_{i_{0}}\left(\boldsymbol{\theta}\right)\boldsymbol{U}_{i_{0}i_{1}\cdots i_{j}}^{j}\left(\boldsymbol{\theta}\right)-\boldsymbol{U}_{i_{0}i_{1}\cdots i_{j}}^{j}\left(\boldsymbol{\theta}+\boldsymbol{\omega}\right)\left[\boldsymbol{\Lambda}_{i_{1}}\left(\boldsymbol{\theta}\right)\otimes\cdots\otimes\boldsymbol{\Lambda}_{i_{j}}\left(\boldsymbol{\theta}\right)\right] & =\boldsymbol{\Gamma}_{i_{0}i_{1}\cdots i_{j}}^{j}\left(\boldsymbol{\theta}\right)\;\text{if}\;\exists i_{k},\;\text{s.t.}\;i_{k}\notin\mathcal{I},\label{eq:FOIL-poly-decomp-ext}
\end{align}
where
\[
\boldsymbol{\Gamma}_{i_{0}i_{1}\cdots i_{j}}^{j}\left(\boldsymbol{\theta}\right)=\boldsymbol{\Gamma}^{j}\left(\boldsymbol{\theta}\right)\boldsymbol{W}_{i_{1}}\left(\boldsymbol{\theta}\right)\otimes\cdots\otimes\boldsymbol{W}_{i_{j}}\left(\boldsymbol{\theta}\right).
\]
The only difference between (\ref{eq:FOIL-poly-decomp-int}) and (\ref{eq:FOIL-poly-decomp-ext})
is the term $\boldsymbol{R}_{i_{0}i_{1}\cdots i_{j}}^{j}$, and therefore
we treat the two equations together. For now, we set $\boldsymbol{R}_{i_{0}i_{1}\cdots i_{j}}^{j}=\boldsymbol{0}$
and solve for $\boldsymbol{U}_{i_{0}i_{1}\cdots i_{j}}^{j}$ in both
equations as long as it is possible.

To bring (\ref{eq:FOIL-poly-decomp-ext}) into a more convenient
form, we multiply (\ref{eq:FOIL-poly-decomp-ext}) by $\boldsymbol{\Lambda}_{i_{0}}^{-1}\left(\boldsymbol{\theta}\right)$
from the left and find
\begin{equation}
\boldsymbol{U}_{i_{0}i_{1}\cdots i_{j}}^{j}\left(\boldsymbol{\theta}\right)-\boldsymbol{\Lambda}_{i_{0}}^{-1}\left(\boldsymbol{\theta}\right)\boldsymbol{U}_{i_{0}i_{1}\cdots i_{j}}^{j}\left(\boldsymbol{\theta}+\boldsymbol{\omega}\right)\left[\boldsymbol{\Lambda}_{i_{1}}\left(\boldsymbol{\theta}\right)\otimes\cdots\otimes\boldsymbol{\Lambda}_{i_{j}}\left(\boldsymbol{\theta}\right)\right]=\boldsymbol{\Lambda}_{i_{0}}^{-1}\left(\boldsymbol{\theta}\right)\boldsymbol{\Gamma}_{i_{0}i_{1}\cdots i_{j}}^{j}\left(\boldsymbol{\theta}\right).\label{eq:FOIL-component-homological}
\end{equation}
Equation (\ref{eq:FOIL-component-homological}) is still in a cumbersome
form, therefore we introduce the vectorisation $\vec{\boldsymbol{u}}=\mathrm{vec}\left(\boldsymbol{U}_{i_{0}i_{1}\cdots i_{j}}^{j}\right)$
and define operator $\boldsymbol{B}$ and its inverse by 
\begin{align*}
\boldsymbol{B}\left(\boldsymbol{\theta}\right)\vec{\boldsymbol{u}}\left(\boldsymbol{\theta}\right) & =\mathrm{vec}\left(\boldsymbol{\Lambda}_{i_{0}}^{-1}\left(\boldsymbol{\theta}\right)\boldsymbol{U}_{i_{0}i_{1}\cdots i_{j}}^{j}\left(\boldsymbol{\theta}\right)\left[\boldsymbol{\Lambda}_{i_{1}}\left(\boldsymbol{\theta}\right)\otimes\cdots\otimes\boldsymbol{\Lambda}_{i_{j}}\left(\boldsymbol{\theta}\right)\right]\right),\\
\boldsymbol{B}^{-1}\left(\boldsymbol{\theta}\right)\vec{\boldsymbol{u}}\left(\boldsymbol{\theta}\right) & =\mathrm{vec}\left(\boldsymbol{\Lambda}_{i_{0}}\left(\boldsymbol{\theta}\right)\boldsymbol{U}_{i_{0}i_{1}\cdots i_{j}}^{j}\left(\boldsymbol{\theta}\right)\left[\boldsymbol{\Lambda}_{i_{1}}^{-1}\left(\boldsymbol{\theta}\right)\otimes\cdots\otimes\boldsymbol{\Lambda}_{i_{j}}^{-1}\left(\boldsymbol{\theta}\right)\right]\right).
\end{align*}
Equation (\ref{eq:FOIL-component-homological}) is now in a simpler
form
\begin{equation}
\vec{\boldsymbol{u}}\left(\boldsymbol{\theta}\right)-\boldsymbol{B}\left(\boldsymbol{\theta}\right)\vec{\boldsymbol{u}}\left(\boldsymbol{\theta}+\boldsymbol{\omega}\right)=\boldsymbol{\eta}\left(\boldsymbol{\theta}\right),\label{eq:FOIL-hom-fw-iter}
\end{equation}
which can be solved by finding the limit of either the forward iteration
\begin{equation}
\vec{\boldsymbol{u}}^{\left(k\right)}\left(\boldsymbol{\theta}\right)=\boldsymbol{B}\left(\boldsymbol{\theta}\right)\vec{\boldsymbol{u}}^{\left(k-1\right)}\left(\boldsymbol{\theta}+\boldsymbol{\omega}\right)+\boldsymbol{\eta}\left(\boldsymbol{\theta}\right)\label{eq:FOIL-hom-fw-iter-id}
\end{equation}
or the inverse iteration
\begin{equation}
\vec{\boldsymbol{u}}^{\left(k\right)}\left(\boldsymbol{\theta}\right)=\boldsymbol{B}^{-1}\left(\boldsymbol{\theta}-\boldsymbol{\omega}\right)\vec{\boldsymbol{u}}^{\left(k-1\right)}\left(\boldsymbol{\theta}-\boldsymbol{\omega}\right)-\boldsymbol{B}^{-1}\left(\boldsymbol{\theta}-\boldsymbol{\omega}\right)\boldsymbol{\eta}\left(\boldsymbol{\theta}-\boldsymbol{\omega}\right).\label{eq:FOIL-hom-bw-iter-id}
\end{equation}
Expanding the iterations (\ref{eq:FOIL-hom-fw-iter-id}) and (\ref{eq:FOIL-hom-bw-iter-id})
yields 
\begin{align}
\vec{\boldsymbol{u}}^{\left(k\right)}\left(\boldsymbol{\theta}\right) & =\boldsymbol{B}\left(\boldsymbol{\theta}\right)\cdots\boldsymbol{B}\left(\boldsymbol{\theta}+\left(k-1\right)\boldsymbol{\omega}\right)\vec{\boldsymbol{u}}^{\left(1\right)}\left(\boldsymbol{\theta}+k\boldsymbol{\omega}\right)+\boldsymbol{B}\left(\boldsymbol{\theta}\right)\cdots\boldsymbol{B}\left(\boldsymbol{\theta}+\left(k-2\right)\boldsymbol{\omega}\right)\boldsymbol{\eta}\left(\boldsymbol{\theta}+\left(k-1\right)\boldsymbol{\omega}\right)+\cdots\nonumber \\
 & \qquad\cdots+\boldsymbol{B}\left(\boldsymbol{\theta}\right)\boldsymbol{\eta}\left(\boldsymbol{\theta}+\boldsymbol{\omega}\right)+\boldsymbol{\eta}\left(\boldsymbol{\theta}\right),\;\text{or}\label{eq:FOIL-hom-fw-iter-exp}\\
\vec{\boldsymbol{u}}^{\left(k\right)}\left(\boldsymbol{\theta}\right) & =\boldsymbol{B}^{-1}\left(\boldsymbol{\theta}-\boldsymbol{\omega}\right)\cdots\boldsymbol{B}^{-1}\left(\boldsymbol{\theta}-k\boldsymbol{\omega}\right)\vec{\boldsymbol{u}}^{\left(1\right)}\left(\boldsymbol{\theta}-k\boldsymbol{\omega}\right)-\boldsymbol{B}^{-1}\left(\boldsymbol{\theta}-\boldsymbol{\omega}\right)\cdots\boldsymbol{B}^{-1}\left(\boldsymbol{\theta}-k\boldsymbol{\omega}\right)\boldsymbol{\eta}\left(\boldsymbol{\theta}-k\boldsymbol{\omega}\right)-\cdots\nonumber \\
 & \qquad\cdots-\boldsymbol{B}^{-1}\left(\boldsymbol{\theta}-\boldsymbol{\omega}\right)\boldsymbol{\eta}\left(\boldsymbol{\theta}-\boldsymbol{\omega}\right).\label{eq:FOIL-hom-bw-iter-exp}
\end{align}
First we find out the condition for convergence of (\ref{eq:FOIL-hom-fw-iter-exp})
based on the exponential dichotomies of $\boldsymbol{\Lambda}_{i}$
and that requires us to calculate the bounds on the term 
\begin{multline}
\boldsymbol{B}\left(\boldsymbol{\theta}\right)\cdots\boldsymbol{B}\left(\boldsymbol{\theta}+\left(k-1\right)\boldsymbol{\omega}\right)\vec{\boldsymbol{u}}\left(\boldsymbol{\theta}\right)=\mathrm{vec}\Bigg(\boldsymbol{\Lambda}_{i_{0}}^{-1}\left(\boldsymbol{\theta}\right)\cdots\boldsymbol{\Lambda}_{i_{0}}^{-1}\left(\boldsymbol{\theta}+\left(k-1\right)\boldsymbol{\omega}\right)\boldsymbol{U}_{i_{0}i_{1}\cdots i_{j}}^{j}\left(\boldsymbol{\theta}\right)\times\\
\times\left[\boldsymbol{\Lambda}_{i_{1}}\left(\boldsymbol{\theta}+\left(k-1\right)\boldsymbol{\omega}\right)\cdots\boldsymbol{\Lambda}_{i_{1}}\left(\boldsymbol{\theta}\right)\otimes\cdots\otimes\boldsymbol{\Lambda}_{i_{j}}\left(\boldsymbol{\theta}+\left(k-1\right)\boldsymbol{\omega}\right)\cdots\boldsymbol{\Lambda}_{i_{j}}\left(\boldsymbol{\theta}\right)\right]\Bigg).\label{eq:FOIL-B-fw-iter}
\end{multline}
Using the properties of tensor products and the exponential bounds
(\ref{eq:ED-straight-bound}) and (\ref{eq:ED-inverse-bound}) we
find that 
\[
\left|\boldsymbol{B}\left(\boldsymbol{\theta}\right)\cdots\boldsymbol{B}\left(\boldsymbol{\theta}+\left(k-1\right)\boldsymbol{\omega}\right)\right|\le C_{i_{0}}C_{i_{1}}\cdots C_{i_{j}}\tilde{\alpha}_{i_{0}}^{-k}\tilde{\beta}_{i_{1}}^{k}\cdots\tilde{\beta}_{i_{j}}^{k}=C_{i_{0}}C_{i_{1}}\cdots C_{i_{j}}\left(\tilde{\alpha}_{i_{0}}^{-1}\tilde{\beta}_{i_{1}}\cdots\tilde{\beta}_{i_{j}}\right)^{k},
\]
for $\tilde{\alpha}_{i_{0}}<\alpha_{i_{0}}$, $\tilde{\beta}_{i_{k}}>\beta_{i_{k}}$,
where $k=1\ldots j$. For the inverse iteration we investigate the
term 
\begin{multline*}
\boldsymbol{B}^{-1}\left(\boldsymbol{\theta}-\boldsymbol{\omega}\right)\cdots\boldsymbol{B}^{-1}\left(\boldsymbol{\theta}-k\boldsymbol{\omega}\right)\vec{\boldsymbol{u}}\left(\boldsymbol{\theta}-k\boldsymbol{\omega}\right)=\mathrm{vec}\Bigg(\boldsymbol{\Lambda}_{i_{0}}\left(\boldsymbol{\theta}-\boldsymbol{\omega}\right)\cdots\boldsymbol{\Lambda}_{i_{0}}\left(\boldsymbol{\theta}-k\boldsymbol{\omega}\right)\boldsymbol{U}_{i_{0}i_{1}\cdots i_{j}}^{j}\left(\boldsymbol{\theta}-k\boldsymbol{\omega}\right)\times,\\
\times\left[\boldsymbol{\Lambda}_{i_{1}}^{-1}\left(\boldsymbol{\theta}-k\boldsymbol{\omega}\right)\cdots\boldsymbol{\Lambda}_{i_{1}}^{-1}\left(\boldsymbol{\theta}-\boldsymbol{\omega}\right)\otimes\cdots\otimes\boldsymbol{\Lambda}_{i_{j}}^{-1}\left(\boldsymbol{\theta}-k\boldsymbol{\omega}\right)\cdots\boldsymbol{\Lambda}_{i_{j}}^{-1}\left(\boldsymbol{\theta}-\boldsymbol{\omega}\right)\right]\Bigg),
\end{multline*}
whose bound can be found in the same ways from equations (\ref{eq:ED-straight-bound})
and (\ref{eq:ED-inverse-bound}) as 
\[
\left|\boldsymbol{B}^{-1}\left(\boldsymbol{\theta}-\boldsymbol{\omega}\right)\cdots\boldsymbol{B}^{-1}\left(\boldsymbol{\theta}-k\boldsymbol{\omega}\right)\right|\le C_{i_{0}}C_{i_{1}}\cdots C_{i_{j}}\tilde{\beta}_{i_{0}}^{k}\tilde{\alpha}_{i_{1}}^{-k}\cdots\tilde{\alpha}_{i_{j}}^{-k}=C_{i_{0}}C_{i_{1}}\cdots C_{i_{j}}\left(\tilde{\beta}_{i_{0}}^{-1}\tilde{\alpha}_{i_{1}}\cdots\tilde{\alpha}_{i_{j}}\right)^{-k},
\]
where $\tilde{\beta}_{i_{0}}>\beta_{i_{0}}$, $\tilde{\alpha}_{i_{k}}<\alpha_{i_{k}}$
for $k=1\ldots j$. This means that either of the series converges
to a unique analytic function if $\beta_{i_{0}}^{-1}\alpha_{i_{1}}\cdots\alpha_{i_{j}}>1$
or $\alpha_{i_{0}}^{-1}\beta_{i_{1}}\cdots\beta_{i_{j}}<1$. In summary,
we can solve equation (\ref{eq:FOIL-component-homological}) if $1\notin\left[\beta_{i_{0}}^{-1}\alpha_{i_{1}}\cdots\alpha_{i_{j}},\alpha_{i_{0}}^{-1}\beta_{i_{1}}\cdots\beta_{i_{j}},\right]$,
which is called our non-resonance condition (\ref{eq:FOIL-non-resonance-1}).

We now discuss the case of internal resonances, i.e., when (\ref{eq:FOIL-poly-decomp-int})
cannot be solved with $\boldsymbol{R}_{i_{0}i_{1}\cdots i_{j}}^{j}=\boldsymbol{0}$,
because $1\in\left[\beta_{i_{0}}^{-1}\alpha_{i_{1}}\cdots\alpha_{i_{j}},\alpha_{i_{0}}^{-1}\beta_{i_{1}}\cdots\beta_{i_{j}}\right]$
for an index set $i_{0},i_{1},\ldots,i_{j}\in\mathcal{I}$. In this
case we set $\boldsymbol{U}_{i_{0}i_{1}\cdots i_{j}}^{j}=\boldsymbol{0}$
and $\boldsymbol{R}_{i_{0}i_{1}\cdots i_{j}}^{j}=\boldsymbol{\Gamma}_{i_{0}i_{1}\cdots i_{j}}^{j}$
in (\ref{eq:FOIL-poly-decomp-int}), which is a solution of (\ref{eq:FOIL-poly-decomp-int}).

We note that no resonances are possible if either a) $\beta_{i_{0}}^{-1}\alpha_{i_{1}}\cdots\alpha_{i_{j}}>1$
or b) $\alpha_{i_{0}}^{-1}\beta_{i_{1}}\cdots\beta_{i_{j}}<1$. In
case a) we must have
\[
j\min\log\alpha_{i}-\max_{i\in\mathcal{I}}\log\beta_{i}>0,
\]
which implies the sub-cases
\begin{align}
j & >\frac{\max_{i\in\mathcal{I}}\log\beta_{i}}{\min\log\alpha_{i}}\;\text{if}\;\min\log\alpha_{i}>0,\label{eq:FOIL-sub-a}\\
j & <\frac{\max_{i\in\mathcal{I}}\log\beta_{i}}{\min\log\alpha_{i}}\;\text{if}\;\min\log\alpha_{i}<0.\nonumber 
\end{align}
In case b), we must have
\[
-\min_{i\in\mathcal{I}}\log\alpha_{i_{0}}+j\max\log\beta_{i}<0,
\]
which implies the sub-cases
\begin{align}
j & <\frac{\min_{i\in\mathcal{I}}\log\alpha_{i_{0}}}{\max\log\beta_{i}}\;\text{if}\;\max\log\beta_{i}>0,\nonumber \\
j & >\frac{\min_{i\in\mathcal{I}}\log\alpha_{i_{0}}}{\max\log\beta_{i}}\;\text{if}\;\max\log\beta_{i}<0.\label{eq:FOIL-sub-b}
\end{align}
Only sub-cases (\ref{eq:FOIL-sub-a}) and (\ref{eq:FOIL-sub-b}) make
sense, because we want to have a lower bound on polynomial orders
for which there are no resonances. We also note that (\ref{eq:FOIL-sub-a})
and (\ref{eq:FOIL-sub-b}) are identical if we replace $\boldsymbol{F}$
(and therefore $\boldsymbol{A}$) by its inverse. Therefore we only
focus on the case when the invariant torus is an attractor, that is
(\ref{eq:FOIL-sub-b}). We also notice that (\ref{eq:FOIL-sub-b})
is the same as $j>\beth_{\mathcal{I}}$ stipulated in our theorem.

What remains is to find the correction term $\boldsymbol{U}^{>}$.
We denote the truncated expansion by 
\[
\boldsymbol{U}^{\le}\left(\boldsymbol{x},\boldsymbol{\theta}\right)=\sum_{j=1}^{\sigma}\boldsymbol{U}^{j}\left(\boldsymbol{\theta}\right)\boldsymbol{x}^{\otimes j}
\]
and this way we can write the invariance equation (\ref{eq:FOIL-invariance})
as
\begin{equation}
\boldsymbol{R}\left(\boldsymbol{U}^{\le}\left(\boldsymbol{x},\boldsymbol{\theta}\right)+\boldsymbol{U}^{>}\left(\boldsymbol{x},\boldsymbol{\theta}\right),\boldsymbol{\theta}\right)=\boldsymbol{U}^{\le}\left(\boldsymbol{F}\left(\boldsymbol{x},\boldsymbol{\theta}\right),\boldsymbol{\theta}+\boldsymbol{\omega}\right)+\boldsymbol{U}^{>}\left(\boldsymbol{F}\left(\boldsymbol{x},\boldsymbol{\theta}\right),\boldsymbol{\theta}+\boldsymbol{\omega}\right).\label{eq:FOIL-Ug-pre}
\end{equation}
Applying the inverse $\boldsymbol{R}^{-1}$ to both sides of (\ref{eq:FOIL-Ug-pre}),
yields 
\[
\boldsymbol{U}^{\le}\left(\boldsymbol{x},\boldsymbol{\theta}\right)+\boldsymbol{U}^{>}\left(\boldsymbol{x},\boldsymbol{\theta}\right)=\boldsymbol{R}^{-1}\left(\boldsymbol{U}^{\le}\left(\boldsymbol{F}\left(\boldsymbol{x},\boldsymbol{\theta}\right),\boldsymbol{\theta}+\boldsymbol{\omega}\right)+\boldsymbol{U}^{>}\left(\boldsymbol{F}\left(\boldsymbol{x},\boldsymbol{\theta}\right),\boldsymbol{\theta}+\boldsymbol{\omega}\right),\boldsymbol{\theta}\right),
\]
which translates to finding a root of the operator
\begin{equation}
\left(\mathcal{H}\left(\boldsymbol{U}^{>}\right)\right)\left(\boldsymbol{x},\boldsymbol{\theta}\right)=\boldsymbol{U}^{\le}\left(\boldsymbol{x},\boldsymbol{\theta}\right)+\boldsymbol{U}^{>}\left(\boldsymbol{x},\boldsymbol{\theta}\right)-\boldsymbol{R}^{-1}\left(\boldsymbol{U}^{\le}\left(\boldsymbol{F}\left(\boldsymbol{x},\boldsymbol{\theta}\right),\boldsymbol{\theta}+\boldsymbol{\omega}\right)+\boldsymbol{U}^{>}\left(\boldsymbol{F}\left(\boldsymbol{x},\boldsymbol{\theta}\right),\boldsymbol{\theta}+\boldsymbol{\omega}\right),\boldsymbol{\theta}\right).\label{eq:FOIL-root-of}
\end{equation}
To aid this endeavour, we introduce a new norm on the space of analytic
functions
\[
\left\Vert \boldsymbol{\Delta}\right\Vert _{\sigma,\rho}=\sup_{\left|\boldsymbol{x}\right|\le\rho,\boldsymbol{\theta}\in\mathbb{T}^{d}}\left|\boldsymbol{x}\right|^{-\sigma}\left|\boldsymbol{\Delta}\left(\boldsymbol{x},\boldsymbol{\theta}\right)\right|,
\]
which defines the Banach space 
\[
X_{\sigma,\rho}=\left\{ \boldsymbol{\Delta}\in C^{a}\left(B_{\rho}\left(\boldsymbol{0}\right)\times\mathbb{T}^{d},X\right):\left\Vert \boldsymbol{\boldsymbol{\Delta}}\right\Vert _{\sigma,\rho}<\infty\right\} .
\]

We use theorem \ref{thm:Kantorovich} to establish the existence of
a unique root of $\mathcal{H}$. First we calculate the quantities
\begin{align*}
\mathcal{H}\left(\boldsymbol{0}\right)\left(\boldsymbol{x},\boldsymbol{\theta}\right) & =\boldsymbol{U}^{\le}\left(\boldsymbol{x},\boldsymbol{\theta}\right)-\boldsymbol{R}^{-1}\left(\boldsymbol{U}^{\le}\left(\boldsymbol{F}\left(\boldsymbol{x},\boldsymbol{\theta}\right),\boldsymbol{\theta}+\boldsymbol{\omega}\right),\boldsymbol{\theta}\right),\\
\left(D\mathcal{H}\left(\boldsymbol{U}^{>}\right)\boldsymbol{\Delta}\right)\left(\boldsymbol{x},\boldsymbol{\theta}\right) & =\boldsymbol{\Delta}\left(\boldsymbol{x},\boldsymbol{\theta}\right)-D_{1}\boldsymbol{R}^{-1}\left(\boldsymbol{U}^{\le}\left(\boldsymbol{F}\left(\boldsymbol{x},\boldsymbol{\theta}\right),\boldsymbol{\theta}+\boldsymbol{\omega}\right)+\boldsymbol{U}^{>}\left(\boldsymbol{F}\left(\boldsymbol{x},\boldsymbol{\theta}\right),\boldsymbol{\theta}+\boldsymbol{\omega}\right),\boldsymbol{\theta}\right)\times\\
 & \qquad\times\boldsymbol{\Delta}\left(\boldsymbol{F}\left(\boldsymbol{x},\boldsymbol{\theta}\right),\boldsymbol{\theta}+\boldsymbol{\omega}\right).
\end{align*}
Due to the polynomial approximation of the actual solution by $\boldsymbol{U}^{\le}$,
for all $\epsilon>0$ there is a $\rho$ such that $\left|\mathcal{H}\left(\boldsymbol{0}\right)\right|<\epsilon$.
We also notice that $D\mathcal{H}\left(\boldsymbol{U}^{>}\right)$
is Lipschitz continuous in the operator norm induced by $\left\Vert \cdot\right\Vert _{\sigma,\rho}$
and therefore the Lipschitz constant $\nu$ is finite. We now establish
that 
\[
D\mathcal{H}\left(\boldsymbol{0}\right)\boldsymbol{\Delta}=\boldsymbol{\Delta}-\mathcal{H}_{0}\boldsymbol{\Delta}
\]
is invertible, where
\[
\left(\mathcal{H}_{0}\boldsymbol{\Delta}\right)\left(\boldsymbol{x},\boldsymbol{\theta}\right)=D_{1}\boldsymbol{R}^{-1}\left(\boldsymbol{U}^{\le}\left(\boldsymbol{F}\left(\boldsymbol{x},\boldsymbol{\theta}\right),\boldsymbol{\theta}+\boldsymbol{\omega}\right),\boldsymbol{\theta}\right)\boldsymbol{\Delta}\left(\boldsymbol{F}\left(\boldsymbol{x},\boldsymbol{\theta}\right),\boldsymbol{\theta}+\boldsymbol{\omega}\right).
\]
Again, due to the polynomial approximation, $\mathcal{T}_{0}$ is
arbitrarily close to 
\[
\left(\mathcal{\overline{H}}_{0}\boldsymbol{\Delta}\right)\left(\boldsymbol{x},\boldsymbol{\theta}\right)=D_{1}\boldsymbol{R}^{-1}\left(\boldsymbol{0},\boldsymbol{\theta}\right)\boldsymbol{\Delta}\left(\boldsymbol{A}\left(\boldsymbol{\theta}\right)\boldsymbol{x},\boldsymbol{\theta}+\boldsymbol{\omega}\right)
\]
for sufficiently small $\rho$. Therefore if $\mathcal{\overline{H}}_{0}$
is a contraction, there is a $\rho>0$ such that $\mathcal{H}_{0}$
is a contraction, too. We can only establish whether $\mathcal{\overline{H}}_{0}$
is a contraction after a number of iterations, and therefore we calculate
\[
\left(\mathcal{\overline{H}}_{0}^{k}\boldsymbol{\Delta}\right)\left(\boldsymbol{x},\boldsymbol{\theta}\right)=D_{1}\boldsymbol{R}^{-1}\left(\boldsymbol{0},\boldsymbol{\theta}\right)\cdots D_{1}\boldsymbol{R}^{-1}\left(\boldsymbol{0},\boldsymbol{\theta}+\left(k-1\right)\boldsymbol{\omega}\right)\boldsymbol{\Delta}\left(\boldsymbol{A}\left(\boldsymbol{\theta}\right)\cdots\boldsymbol{A}\left(\boldsymbol{\theta}+\left(k-1\right)\boldsymbol{\omega}\right)\boldsymbol{x},\boldsymbol{\theta}+k\boldsymbol{\omega}\right).
\]
Calculating 
\begin{multline}
\left\Vert \mathcal{\overline{H}}_{0}^{k}\boldsymbol{\Delta}\right\Vert _{\sigma,\rho}=\sup_{\left|\boldsymbol{A}^{-1}\left(\boldsymbol{\theta}+\left(k-1\right)\boldsymbol{\omega}\right)\cdots\boldsymbol{A}^{-1}\left(\boldsymbol{\theta}\right)\boldsymbol{y}\right|\le\rho,\boldsymbol{\theta}\in\mathbb{T}^{d}}\left|\boldsymbol{A}^{-1}\left(\boldsymbol{\theta}+\left(k-1\right)\boldsymbol{\omega}\right)\cdots\boldsymbol{A}^{-1}\left(\boldsymbol{\theta}\right)\boldsymbol{y}\right|^{-\sigma}\times\\
\times\left|D_{1}\boldsymbol{R}^{-1}\left(\boldsymbol{0},\boldsymbol{\theta}\right)\cdots D_{1}\boldsymbol{R}^{-1}\left(\boldsymbol{0},\boldsymbol{\theta}+\left(k-1\right)\boldsymbol{\omega}\right)\boldsymbol{\Delta}\left(\boldsymbol{y},\boldsymbol{\theta}+k\boldsymbol{\omega}\right)\right|\label{eq:FOIL-contr-norm}
\end{multline}
takes a number of steps. Using (\ref{eq:ED-inverse-lower}) we estimate
that 
\[
\left|\boldsymbol{A}^{-1}\left(\boldsymbol{\theta}+\left(k-1\right)\boldsymbol{\omega}\right)\cdots\boldsymbol{A}^{-1}\left(\boldsymbol{\theta}\right)\right|\ge C_{m}\tilde{\beta}^{-k},
\]
for $\tilde{\beta}>\beta_{m}<1$. Given that $C_{m}\beta_{m}^{-k}>1$
for sufficiently large $k$, we concluded that if 
\[
\left|\boldsymbol{A}^{-1}\left(\boldsymbol{\theta}+\left(k-1\right)\boldsymbol{\omega}\right)\cdots\boldsymbol{A}^{-1}\left(\boldsymbol{\theta}\right)\boldsymbol{y}\right|\le\rho,
\]
so is $\left|\boldsymbol{y}\right|\le\rho$. This implies that the
norm (\ref{eq:FOIL-contr-norm}) is bounded by 
\begin{align*}
\left\Vert \mathcal{\overline{H}}_{0}\boldsymbol{\Delta}\right\Vert _{\sigma,\rho} & \le\sup_{\left|\boldsymbol{y}\right|\le\rho,\boldsymbol{\theta}\in\mathbb{T}^{d}}C_{m}^{-\sigma}\beta_{m}^{k\sigma}\left|\boldsymbol{y}\right|^{-\sigma}\left|D_{1}\boldsymbol{R}^{-1}\left(\boldsymbol{0},\boldsymbol{\theta}\right)\cdots D_{1}\boldsymbol{R}^{-1}\left(\boldsymbol{0},\boldsymbol{\theta}+\left(k-1\right)\boldsymbol{\omega}\right)\boldsymbol{\Delta}\left(\boldsymbol{y},\boldsymbol{\theta}+k\boldsymbol{\omega}\right)\right|,\\
 & \le\sup_{\left|\boldsymbol{y}\right|\le\rho,\boldsymbol{\theta}\in\mathbb{T}^{d}}C_{m}^{-\sigma}\beta_{m}^{k\sigma}\left|\boldsymbol{y}\right|^{-\sigma}\left(\min_{i\in\mathcal{I}}\alpha_{i}\right)^{-k}\left|\boldsymbol{\Delta}\left(\boldsymbol{y},\boldsymbol{\theta}+k\boldsymbol{\omega}\right)\right|,\\
 & =\sup_{\left|\boldsymbol{y}\right|\le\rho,\boldsymbol{\theta}\in\mathbb{T}^{d}}C_{m}^{-\sigma}\beta_{m}^{k\sigma}\left|\boldsymbol{y}\right|^{-\sigma}\left(\min_{i\in\mathcal{I}}\alpha_{i}\right)^{-k}\left|\boldsymbol{\Delta}\left(\boldsymbol{y},\boldsymbol{\theta}\right)\right|,\\
 & =C_{m}^{-\sigma}\beta_{m}^{k\sigma}\left(\min_{i\in\mathcal{I}}\alpha_{i}\right)^{-k}\left\Vert \boldsymbol{\Delta}\right\Vert _{\sigma,\rho}.
\end{align*}
Now all we need is that 
\begin{align*}
\beta_{m}^{\sigma}\left(\min_{i\in\mathcal{I}}\alpha_{i}\right)^{-1} & <1\\
\sigma\log\beta_{m}-\min_{i\in\mathcal{I}}\log\alpha_{i} & <0\\
\sigma\log\beta_{m} & <\min_{i\in\mathcal{I}}\log\alpha_{i}\\
\sigma> & \frac{\min_{i\in\mathcal{I}}\log\alpha_{i}}{\log\beta_{m}}=\beth_{\mathcal{I}},
\end{align*}
and then for $k$ sufficiently large that $C_{m}^{-\sigma}\beta_{m}^{k\sigma}\left(\min_{i\in\mathcal{I}}\alpha_{i}\right)^{-k}<1$.
This the implies that $\mathcal{\overline{H}}_{0}$ is a contraction
and so is $\mathcal{H}_{0}$ for $\rho>0$ sufficiently small. This
also means that $\left\Vert D\mathcal{H}\left(\boldsymbol{0}\right)\right\Vert \le\mu<\infty$
and that $\left\Vert D\mathcal{H}\left(\boldsymbol{0}\right)\mathcal{H}\left(\boldsymbol{0}\right)\right\Vert \le\left\Vert D\mathcal{H}\left(\boldsymbol{0}\right)\right\Vert \left\Vert \mathcal{H}\left(\boldsymbol{0}\right)\right\Vert =\mu\epsilon=:\lambda$
is arbitrarily small. As a result we can apply theorem (\ref{thm:Kantorovich})
and find that we have a unique solution for $\boldsymbol{U}^{>}$.

The last thing to show is uniqueness. In our calculation we have shown
that for any permissible choice of $\boldsymbol{R}$, we have a unique
$\boldsymbol{U}$. It is also clear that between two permissible $\boldsymbol{R}$
and $\hat{\boldsymbol{R}}$ there is always a transformation $\boldsymbol{T}$,
such that 
\begin{equation}
\hat{\boldsymbol{R}}\left(\boldsymbol{\theta},\boldsymbol{T}\left(\boldsymbol{\theta},\boldsymbol{z}\right)\right)=\boldsymbol{T}\left(\boldsymbol{\theta}+\boldsymbol{\omega},\boldsymbol{R}\left(\boldsymbol{\theta},\boldsymbol{z}\right)\right),\label{eq:FOIL-unique-tr}
\end{equation}
because equation (\ref{eq:FOIL-unique-tr}) has the exact same non-resonance
conditions and choices for $\hat{\boldsymbol{R}}$ as the invariance
equation (\ref{eq:FOIL-invariance}). Now we notice that $\hat{\boldsymbol{U}}=\boldsymbol{T}\circ\boldsymbol{U}$
together with $\hat{\boldsymbol{R}}$ also satisfies the invariance
equation (\ref{eq:FOIL-invariance}) and that $\hat{\boldsymbol{U}}$
is the unique encoder, given $\hat{\boldsymbol{R}}$. Therefore we
conclude that any two solutions of (\ref{eq:FOIL-invariance}) only
differ by coordinate transformation $\boldsymbol{T}$. This makes
the invariant foliation $\mathcal{F}_{\mathcal{I}}$ a unique geometric
object.
\end{proof}

\section{Numerical implementation}

We use truncated Fourier series to represent functions along the torus
$\mathbb{T}$, so that a function is given by
\[
\boldsymbol{x}\left(\theta\right)=\sum_{j=-\ell}^{\ell}\boldsymbol{x}_{j}\mathrm{e}^{ij\theta},
\]
where $\overline{\boldsymbol{x}}_{j}=\boldsymbol{x}_{-j}$ are the
values we hold in the computer memory. We also need to use the shift
operator 
\[
\left(\mathcal{S}^{\omega}\boldsymbol{x}\right)\left(\theta\right)=\boldsymbol{x}\left(\theta-\omega\right),
\]
which then implies that
\[
\left(\mathcal{S}^{\omega}\boldsymbol{x}\right)\left(\theta\right)=\sum_{j=-\ell}^{\ell}\boldsymbol{x}_{j}\mathrm{e}^{ij\left(\theta-\omega\right)}=\sum_{j=-\ell}^{\ell}\boldsymbol{x}_{j}\mathrm{e}^{-ij\omega}\mathrm{e}^{ij\theta}
\]
and therefore the discretised shift operator is the diagonal matrix
\[
\mathbb{S}_{jk}^{\omega}=\delta_{jk}\mathrm{e}^{-ij\omega}.
\]
Our map $\boldsymbol{F}$ is given as a Fourier series $\boldsymbol{F}\left(\boldsymbol{x},\theta\right)=\sum_{j=-\ell}^{\ell}\boldsymbol{F}_{j}\left(\boldsymbol{x}\right)\mathrm{e}^{ij\theta}$,
and therefore the invariant torus is a solution of 
\[
\boldsymbol{K}_{k}\mathrm{e}^{ik\omega}=\boldsymbol{F}_{k}\left(\sum_{j=-\ell}^{\ell}\boldsymbol{K}_{j}\mathrm{e}^{ij\theta}\right),k=-\ell\ldots\ell,
\]
which are well-defined algebraic equations if the torus itself is
normally hyperbolic. 

\subsection{\label{subsec:Bundles}Invariant vector bundles}

To find the invariant vector bundles, we solve equation (\ref{eq:ED-left-invariance})
in its discrete form and also with a constant and diagonal $\boldsymbol{\Lambda}$.
A constant $\boldsymbol{\Lambda}$ is only allowed due to discretisation,
otherwise we would need to consider irreducibility. Diagonality of
$\boldsymbol{\Lambda}$ is also due to the numerical methods that
ignore generalised eigenvectors. The discretised spectral equation
is
\begin{equation}
\sum_{j}\boldsymbol{u}_{j}\boldsymbol{A}_{k-j}\mathrm{e}^{ij\omega}=\lambda\boldsymbol{u}_{k},\label{eq:NUM-ev-left}
\end{equation}
where $\lambda$ are the diagonal elements of $\boldsymbol{\Lambda}$.
Equation (\ref{eq:NUM-ev-left}) is an eigenvector-eigenvalue problem.
We also consider the eigenvalue problem related to (\ref{eq:ED-right-invariance})
in the form of
\begin{equation}
\lambda\boldsymbol{w}_{k}=\mathrm{e}^{-ik\omega}\boldsymbol{A}_{k-j}\boldsymbol{w}_{j},\label{eq:NUM-ev-right}
\end{equation}
which has the same eigenvalues as (\ref{eq:NUM-ev-left}).

The eigenvalues are placed approximately along concentric circles
and we expect that each circle contains $2\ell+1$ or $2\left(2\ell+1\right)$
eigenvalues, depending on whether they can be represented by a real
or a pair of complex conjugate eigenvalues. Using k-nearest neighbours
\cite{biau2015lectures}, we first try to find $n_{cl}=n$ clusters
for the magnitudes of the eigenvalues $\left|\lambda\right|$, expecting
each cluster to have $2\ell+1$ eigenvalues. If this does not hold,
we decrease $n_{cl}$ until each cluster has an integer multiple of
$2\ell+1$ eigenvalues while $n_{cl}\ge n/2$. It is possible to not
find clusters that match our requirement, and there could be many
reasons for that. For example, the value of $\ell$, the number of
Fourier harmonics could be too low, or the actual spectral circles
are too close to each other to resolve them numerically. After a successful
search we have $m$ clusters, denoted by $\mathcal{C}_{l}$ such that
$\bigcup_{l=1}^{m}\mathcal{C}_{l}=\left\{ 1,\ldots,n\left(2\ell+1\right)\right\} $.

We now need to identify a representative eigenvalue for each cluster
of eigenvalues. We pick the one that corresponds to an eigenvector
whose dominant Fourier component has the smallest index in magnitude.
In mathematical formalism, if the eigenvectors are represented by
$\boldsymbol{u}_{j}^{k}$, where $k$ is the index within the cluster
we find the minimising index by
\begin{equation}
k_{l}=\arg\min_{k\in\mathcal{C}_{l}}\sum_{j=-\ell}^{\ell}\left|\boldsymbol{u}_{j}^{k}\right|2^{\left|j\right|}.\label{eq:NUM-evsel}
\end{equation}
Using the identified eigenvectors $\boldsymbol{u}_{j}^{k_{l}}$, we
recover the left and right invariant vector bundles of the linear
system (\ref{eq:ED-linsys}) in the form of
\begin{align*}
\boldsymbol{\Lambda}_{l} & =\lambda_{k_{l}}, & \boldsymbol{U}_{l}^{1}\left(\theta\right) & =\sum_{j=-\ell}^{\ell}\boldsymbol{u}_{j}^{k_{l}}\mathrm{e}^{ij\theta}, & \boldsymbol{W}_{l}^{1}\left(\theta\right) & =\sum_{j=-\ell}^{\ell}\boldsymbol{w}_{j}^{k_{l}}\mathrm{e}^{ij\theta} & \text{if}\;\;\#\mathcal{C}_{l}=\left(2\ell+1\right)\\
\boldsymbol{\Lambda}_{l} & =\begin{pmatrix}\lambda_{k_{l}} & 0\\
0 & \overline{\lambda}_{k_{l}}
\end{pmatrix}, & \boldsymbol{U}_{l}^{1}\left(\theta\right) & =\sum_{j=-\ell}^{\ell}\begin{pmatrix}\boldsymbol{u}_{j}^{k_{l}}\\
\overline{\boldsymbol{u}}_{j}^{k_{l}}
\end{pmatrix}\mathrm{e}^{ij\theta}, & \boldsymbol{W}_{l}^{1}\left(\theta\right) & =\sum_{j=-\ell}^{\ell}\begin{pmatrix}\boldsymbol{w}_{j}^{k_{l}} & \overline{\boldsymbol{w}}_{j}^{k_{l}}\end{pmatrix}\mathrm{e}^{ij\theta} & \text{if}\;\;\#\mathcal{C}_{l}=2\left(2\ell+1\right).
\end{align*}
Finally, the coordinate transformation that brings the nonlinear system
(\ref{eq:MAPSysDef}) into a form with diagonal and constant linear
part is 
\begin{align*}
\boldsymbol{\Lambda} & =\begin{pmatrix}\boldsymbol{\Lambda}_{1} &  & \boldsymbol{0}\\
 & \ddots\\
\boldsymbol{0} &  & \boldsymbol{\Lambda}_{m}
\end{pmatrix} & \boldsymbol{U}^{d}\left(\theta\right) & =\begin{pmatrix}\boldsymbol{U}_{1}^{1}\left(\theta\right)\\
\vdots\\
\boldsymbol{U}_{m}^{1}\left(\theta\right)
\end{pmatrix} & \boldsymbol{W}^{d}\left(\theta\right) & =\begin{pmatrix}\boldsymbol{W}_{1}^{d}\left(\theta\right) &  & \boldsymbol{W}_{m}^{d}\left(\theta\right)\end{pmatrix},
\end{align*}
and the transformed map about the invariant torus is obtained as
\[
\boldsymbol{\Lambda}\boldsymbol{x}+\boldsymbol{N}\left(\boldsymbol{x},\theta\right)=\boldsymbol{U}^{d}\left(\theta+\omega\right)\boldsymbol{F}\left(\boldsymbol{W}^{d}\left(\theta\right)\boldsymbol{x},\theta\right),
\]
where $\boldsymbol{N}\left(\boldsymbol{x},\theta\right)=\mathcal{O}\left(\left|\boldsymbol{x}\right|^{2}\right)$.

\subsection{\label{subsec:FoilCalc}Invariant foliation by series expansion}

To create an invariant foliation we need to pick an index set $\mathcal{I}$
that designates a set of diagonal elements of $\boldsymbol{\Lambda}$,
which should include complex conjugate pairs if a complex eigenvalue
is chosen. To make notation simpler, we assume that the first $\nu$
diagonal elements of $\boldsymbol{\Lambda}$ are the chosen eigenvalues,
as we are free to order them as necessary, thus $\mathcal{I}=\left\{ 1,\ldots,\nu\right\} $.
Using a mixed power and Fourier series expansion we then solve the
invariance equation
\begin{equation}
\boldsymbol{R}\left(\boldsymbol{U}\left(\boldsymbol{x},\theta\right),\theta\right)-\boldsymbol{U}\left(\boldsymbol{\Lambda}\boldsymbol{x}+\boldsymbol{N}\left(\boldsymbol{x},\theta\right),\theta\right)=\boldsymbol{0}.\label{eq:FOIL-power-invariance}
\end{equation}
The ansatz is a mixed trigonometric power series
\begin{align*}
\boldsymbol{U}\left(\boldsymbol{\theta},\boldsymbol{x}\right) & =\sum_{j=1}^{\sigma}\boldsymbol{U}^{j}\left(\boldsymbol{\theta}\right)\boldsymbol{x}^{\otimes j},\\
\boldsymbol{R}\left(\boldsymbol{\theta},\boldsymbol{x}\right) & =\sum_{j=1}^{\sigma}\boldsymbol{R}^{j}\left(\boldsymbol{\theta}\right)\boldsymbol{x}^{\otimes j},
\end{align*}
where
\begin{align*}
\boldsymbol{R}^{j}\left(\boldsymbol{\theta}\right)\boldsymbol{z}^{\otimes j} & =\sum_{i_{0}\cdots i_{j},\boldsymbol{k}}\boldsymbol{e}_{i_{0}}S_{i_{0}i_{1}\cdots i_{j}}^{j,k}\mathrm{e}^{ik\boldsymbol{\theta}}\left(\boldsymbol{e}_{i_{1}}\cdot\boldsymbol{z}\right)\cdots\left(\boldsymbol{e}_{i_{j}}\cdot\boldsymbol{z}\right),\\
\boldsymbol{U}^{j}\left(\boldsymbol{\theta}\right)\boldsymbol{x}^{\otimes j} & =\sum_{i_{0}\cdots i_{j},\boldsymbol{k}}\boldsymbol{e}_{i_{0}}U_{i_{0}i_{1}\cdots i_{j}}^{j,k}\mathrm{e}^{ik\boldsymbol{\theta}}\left(\boldsymbol{e}_{i_{1}}\cdot\boldsymbol{x}\right)\cdots\left(\boldsymbol{e}_{i_{j}}\cdot\boldsymbol{x}\right),
\end{align*}
and $\boldsymbol{e}_{i}$ are the $i$-th canonical unit vectors of
the right dimension, e.g., $\boldsymbol{e}_{1}=\left(1,0,\cdots\right)$.
Substituting the ansatz into (\ref{eq:FOIL-power-invariance}) and
separating the $j$-th order terms leads to the homological equation

\begin{equation}
\boldsymbol{R}^{1}\left(\boldsymbol{\theta}\right)\boldsymbol{U}^{j}\left(\boldsymbol{\theta}\right)\boldsymbol{x}^{\otimes j}+\boldsymbol{R}^{j}\left(\boldsymbol{\theta}\right)\left(\boldsymbol{U}^{1}\left(\boldsymbol{\theta}\right)\boldsymbol{x}\right)^{\otimes j}-\boldsymbol{U}^{j}\left(\boldsymbol{\theta}+\boldsymbol{\omega}\right)\left(\boldsymbol{\Lambda}\boldsymbol{x}\right)^{\otimes j}=\boldsymbol{\Gamma}^{j}\left(\boldsymbol{\theta}\right)\boldsymbol{x}^{\otimes j},\label{eq:FOIL-power-homological}
\end{equation}
where $\boldsymbol{\Gamma}^{j}\left(\boldsymbol{\theta}\right)$ are
terms composed of $\boldsymbol{R}^{k}$ and $\boldsymbol{U}^{l}$,
$k\neq j$, $l\neq j$. 

We solve the homological equation (\ref{eq:FOIL-power-homological})
in increasing order of $j$. The linear solution can be chosen as
\begin{align*}
\boldsymbol{U}^{1}\left(\theta\right) & =\begin{pmatrix}\boldsymbol{e}_{1} & \boldsymbol{e}_{2} & \cdots & \boldsymbol{e}_{\nu}\end{pmatrix}^{T}\\
\boldsymbol{R}^{1}\left(\theta\right) & =\begin{pmatrix}\boldsymbol{e}_{1} & \boldsymbol{e}_{2} & \cdots & \boldsymbol{e}_{\nu}\end{pmatrix}^{T}\boldsymbol{\Lambda}\begin{pmatrix}\boldsymbol{e}_{1} & \boldsymbol{e}_{2} & \cdots & \boldsymbol{e}_{\nu}\end{pmatrix}.
\end{align*}
The nonlinear terms are determined by equation
\[
\lambda_{i_{0}}U_{i_{0}i_{1}\cdots i_{j}}^{j,k}+R_{i_{0}i_{1}\cdots i_{j}}^{j,k}-\lambda_{i_{1}}\cdots\lambda_{i_{j}}\mathrm{e}^{ik\boldsymbol{\omega}}U_{i_{0}i_{1}\cdots i_{j}}^{j,k}=\boldsymbol{\Gamma}_{i_{0}i_{1}\cdots i_{j}}^{j,k},
\]
which has the solution
\begin{align}
U_{i_{0}i_{1}\cdots i_{j}}^{j,k} & =\frac{1}{\lambda_{i_{0}}-\lambda_{i_{1}}\cdots\lambda_{i_{j}}\mathrm{e}^{ik\omega}}\Gamma_{i_{0}i_{1}\cdots i_{j}}^{j,k} & R_{i_{0}i_{1}\cdots i_{j}}^{j,k} & =0\;\text{or}\label{eq:FOIL-pow-nonres}\\
U_{i_{0}i_{1}\cdots i_{j}}^{j,k} & =0 & R_{i_{0}i_{1}\cdots i_{j}}^{j,k} & =\Gamma_{i_{0}i_{1}\cdots i_{j}}^{j,k}.\label{eq:FOIL-pow-res}
\end{align}
It is only possible to choose solution (\ref{eq:FOIL-pow-nonres})
if any of the indices $i_{1}\cdots i_{j}$ are not part of $\mathcal{I}$.
Therefore we must have 
\begin{equation}
\lambda_{i_{0}}-\lambda_{i_{1}}\cdots\lambda_{i_{j}}\mathrm{e}^{ik\omega}\neq0,\label{eq:FOIL-pow-rescon}
\end{equation}
when any of the indices $i_{1}\cdots i_{j}$ are not part of $\mathcal{I}$.
This is the same condition as (\ref{eq:FOIL-non-resonance-1}) expressed
in a discretised fashion. If we have $i_{0}i_{1}\cdots i_{j}\in\mathcal{I}$
the denominator $\lambda_{i_{0}}-\lambda_{i_{1}}\cdots\lambda_{i_{j}}\mathrm{e}^{ik\omega}$
determines whether we choose solution (\ref{eq:FOIL-pow-res}) over
(\ref{eq:FOIL-pow-nonres}).

In this paper we insist to create a $\theta$ independent, i.e, autonomous
$\boldsymbol{R}$ and therefore, we are only allowed to choose (\ref{eq:FOIL-pow-res})
for $k=0$. This excludes parametric resonances, where (\ref{eq:FOIL-pow-rescon})
does not hold for $k\neq0$.

Invariant manifolds can be calculated in the same way, except that
the invariance equation is (\ref{eq:MANIF-MAP-invar}). We do not
detail how this is done, because the invariant manifold can also be
recovered from two complementary invariant foliations as described
in the following section. The same holds true for invariant manifolds
of vector fields, which are calculated by series expanding $D_{1}\boldsymbol{W}\left(\boldsymbol{z},\theta\right)\boldsymbol{R}\left(\boldsymbol{z},\theta\right)+\omega D_{2}\boldsymbol{W}\left(\boldsymbol{z},\theta\right)=\boldsymbol{F}\left(\boldsymbol{W}\left(\boldsymbol{z},\theta\right),\theta\right)$.

\subsection{\label{subsec:Reconstruct}Reconstructing invariant manifolds}

We assume that two foliations are calculated to satisfy the two invariance
equations
\begin{align}
\boldsymbol{R}\left(\boldsymbol{U}\left(\boldsymbol{x},\theta\right),\theta\right) & =\boldsymbol{U}\left(\boldsymbol{F}\left(\boldsymbol{x},\theta\right),\theta+\omega\right),\label{eq:FOIL-inv-P1}\\
\boldsymbol{S}\left(\boldsymbol{V}\left(\boldsymbol{x},\theta\right),\theta\right) & =\boldsymbol{V}\left(\boldsymbol{F}\left(\boldsymbol{x},\theta\right),\theta+\omega\right),\label{eq:FOIL-inv-P2}
\end{align}
such that the sought after invariant manifolds is defined by $\mathcal{M}=\left\{ \left(\boldsymbol{x},\theta\right)\in X\times\mathbb{T}:\boldsymbol{V}\left(\boldsymbol{x},\theta\right)=\boldsymbol{0}\right\} $
with the dynamics represented by the map $\boldsymbol{R}$. To make
sure that the two foliations are complementary, for equation (\ref{eq:FOIL-inv-P2})
we use the index set $\mathcal{I}^{c}=\left\{ 1,\ldots,n\right\} \setminus\mathcal{I}$.
In order to find a manifold immersion $\boldsymbol{W}$ such that
the manifold invariance equation
\begin{equation}
\boldsymbol{W}\left(\boldsymbol{R}\left(\boldsymbol{z},\theta\right),\theta+\omega\right)=\boldsymbol{F}\left(\boldsymbol{W}\left(\boldsymbol{z},\theta\right),\theta\right)\label{eq:MANIF-MAP-invar}
\end{equation}
holds we need to solve the equations
\begin{equation}
\begin{array}{rl}
\boldsymbol{U}\left(\boldsymbol{W}\left(\boldsymbol{z},\theta\right),\theta\right) & =\boldsymbol{z},\\
\boldsymbol{V}\left(\boldsymbol{W}\left(\boldsymbol{z},\theta\right),\theta\right) & =\boldsymbol{0},
\end{array}\label{eq:FOIL-manifold}
\end{equation}
We assume that the encoders are separated into linear and nonlinear
parts, such that $\boldsymbol{U}\left(\boldsymbol{x},\theta\right)=\boldsymbol{U}^{1}\left(\theta\right)\boldsymbol{x}+\boldsymbol{U}^{nl}\left(\boldsymbol{x},\theta\right)$
and $\boldsymbol{V}\left(\boldsymbol{x},\theta\right)=\boldsymbol{V}^{1}\left(\theta\right)\boldsymbol{x}+\boldsymbol{V}^{nl}\left(\boldsymbol{x},\theta\right)$.
The equations (\ref{eq:FOIL-manifold}) can be re-written into the
following form
\[
\begin{pmatrix}\boldsymbol{U}^{1}\left(\theta\right)\\
\boldsymbol{V}^{1}\left(\theta\right)
\end{pmatrix}\boldsymbol{W}=\begin{pmatrix}\boldsymbol{z}-\boldsymbol{U}^{nl}\left(\boldsymbol{W},\theta\right)\\
-\boldsymbol{V}^{nl}\left(\boldsymbol{W},\theta\right)
\end{pmatrix},
\]
and can be solved by iteration
\begin{equation}
\boldsymbol{W}=\begin{pmatrix}\boldsymbol{U}^{1}\left(\theta\right)\\
\boldsymbol{V}^{1}\left(\theta\right)
\end{pmatrix}^{-1}\begin{pmatrix}\boldsymbol{z}-\boldsymbol{U}^{nl}\left(\boldsymbol{W},\theta\right)\\
-\boldsymbol{V}^{nl}\left(\boldsymbol{W},\theta\right)
\end{pmatrix}.\label{eq:FOIL-manif-iter}
\end{equation}

\subsection{\label{subsec:Freq}Frequencies and damping ratios}

As we have discussed previously \cite{Szalai2023Fol}, the dynamics
given by the map $\boldsymbol{R}$ should be interpreted in the nonlinear
coordinate system given by the manifold immersion $\boldsymbol{W}$.
This means that the dynamical properties cannot just be derived from
map $\boldsymbol{R}$, the immersion $\boldsymbol{W}$ also needs
to be considered. Here we summarised how this is done, and refer to
the companion paper \cite{SzalaiForcedData2024}, for detailed derivation.
Here we are concerned about the frequencies and damping ratios of
the vibration that surrounds the invariant torus $\mathcal{T}$ and
contained within the invariant manifolds $\mathcal{M}$. Therefore
we assume that $\boldsymbol{R}$ is autonomous as per the transformation
in section \ref{subsec:FoilCalc}, but $\boldsymbol{W}$remains $\theta$-dependent.
It is also assumed that the manifold $\mathcal{M}$ is two-dimensional
for each point of the torus $\mathcal{T}$ and the linear part if
$\boldsymbol{R}$ has a pair of complex conjugate eigenvalues. Under
these assumptions we can write that
\begin{equation}
\boldsymbol{R}\left(\boldsymbol{z}\right)=\begin{pmatrix}s\left(z,\overline{z}\right)\\
\overline{s}\left(z,\overline{z}\right)
\end{pmatrix},\label{eq:FREQ-R}
\end{equation}
where function $s$ is scalar valued, $z\in\mathbb{C}$ and $\overline{\;}$
means complex conjugation. Given the form (\ref{eq:FREQ-R}), we can
use a polar coordinate system $z=r\mathrm{e}^{i\gamma}$ and therefore
define
\begin{align*}
\widehat{\boldsymbol{W}}\left(r,\gamma,\theta\right) & =\boldsymbol{W}\left(r\mathrm{e}^{i\gamma},r\mathrm{e}^{-i\gamma},\theta\right),\\
R\left(r\right) & =\left|s\left(r\mathrm{e}^{i\gamma},r\mathrm{e}^{-i\gamma}\right)\right|\\
T\left(r\right) & =\arg\,\mathrm{e}^{-i\gamma}s\left(r\mathrm{e}^{i\gamma},r\mathrm{e}^{-i\gamma}\right)
\end{align*}
such that the invariance equation (\ref{eq:FOIL-power-invariance})
becomes
\[
\widehat{\boldsymbol{W}}\left(R\left(r\right),\gamma+T\left(r\right),\theta+\omega\right)=\boldsymbol{F}\left(\widehat{\boldsymbol{W}}\left(r,\gamma,\theta\right),\theta\right).
\]
We now need to create a transformation that makes sure two things:
the amplitude given by $\widehat{\boldsymbol{W}}$ increases linearly
with parameter $r$ and that there is no phase shift between closed
curves defined by the graphs of $\gamma\mapsto\widehat{\boldsymbol{W}}\left(r_{1},\gamma,\theta\right)$
and $\gamma\mapsto\widehat{\boldsymbol{W}}\left(r_{2},\gamma,\theta\right)$
for $r_{1}\neq r_{2}$. The first condition makes sure that damping
ratios are correct, the second makes sure that frequencies are correct.
A lengthy derivation in \cite{Szalai2023Fol} yields that this transformation
is 
\[
\tilde{\boldsymbol{W}}\left(r,\gamma,\boldsymbol{\theta}\right)=\widehat{\boldsymbol{W}}\left(\rho\left(r\right),\gamma+\phi\left(\rho\left(r\right)\right),\boldsymbol{\theta}\right),
\]
where
\begin{align*}
\rho\left(r\right) & =\kappa^{-1}\left(r\right)\\
\phi\left(r\right) & =-\int_{0}^{r}\left[\int_{\mathbb{T}^{d+1}}\left\langle D_{2}\widehat{\boldsymbol{W}}\left(\tilde{r},\gamma,\boldsymbol{\theta}\right),D_{2}\widehat{\boldsymbol{W}}\left(\tilde{r},\gamma,\boldsymbol{\theta}\right)\right\rangle \mathrm{d}\gamma\mathrm{d}\boldsymbol{\theta}\right]^{-1}\int_{\mathbb{T}^{d+1}}\left\langle D_{1}\widehat{\boldsymbol{W}}\left(\tilde{r},\gamma,\boldsymbol{\theta}\right),D_{2}\widehat{\boldsymbol{W}}\left(\tilde{r},\gamma,\boldsymbol{\theta}\right)\right\rangle \mathrm{d}\gamma\mathrm{d}\boldsymbol{\theta}\mathrm{d}\tilde{r}
\end{align*}
with
\[
\kappa\left(\rho\right)=\left(2\pi\right)^{-\left(d+1\right)/2}\sqrt{\int_{\mathbb{T}^{d+1}}\left|\widehat{\boldsymbol{W}}\left(\rho\left(r\right),\gamma,\boldsymbol{\theta}\right)\right|^{2}\mathrm{d}\gamma\mathrm{d}\boldsymbol{\theta}}.
\]
In the new corrected coordinate system the invariance equation becomes
\[
\tilde{\boldsymbol{W}}\left(\tilde{R}\left(r\right),\gamma+\tilde{T}\left(r\right),\boldsymbol{\theta}\right)=\boldsymbol{F}\left(\tilde{\boldsymbol{W}}\left(r,\gamma,\boldsymbol{\theta}\right),\theta\right),
\]
where

\begin{align*}
\tilde{R}\left(r\right) & =\kappa\left(R\left(\rho\left(r\right)\right)\right),\\
\tilde{T}\left(r\right) & =T\left(\rho\left(r\right)\right)+\phi\left(\rho\left(r\right)\right)-\phi\left(R\left(\rho\left(r\right)\right)\right).
\end{align*}
Given that we no longer have amplitude and phase distortion, the instantaneous
frequency and the damping ratio can be calculated as
\begin{align*}
\omega\left(r\right) & =\tilde{T}\left(r\right)/\Delta t,\\
\xi\left(r\right) & =-\log\left(r^{-1}\tilde{R}\left(r\right)\right)/\tilde{T}\left(r\right),
\end{align*}
respectively, where $\Delta t$ is the duration of time, between two
consecutive points of the trajectory generated by $\boldsymbol{F}$.
The ROM in differential equation form becomes
\begin{align*}
\dot{r} & =-\zeta\left(r\right)\omega\left(r\right)r,\\
\dot{\theta} & =\omega\left(r\right).
\end{align*}

\section{Examples}

Here we investigate two examples, which are both two degree-of-freedom
mechanical systems. Due to limitations of polynomial expansions, we
will not consider higher dimensional systems here, those can be found
in the companion paper \cite{SzalaiForcedData2024}, where a more
suitable representation of the submersions $\boldsymbol{U}$ and $\boldsymbol{V}$
is used. The calculations have the following steps
\begin{enumerate}
\item Calculate a discrete-time map from the vector field (ODE) by Taylor
expanding a numerical ODE solver using automatic differentiation.
\item Identify the invariant torus (periodic orbit) from both the ODE and
the generated map. (not detailed here)
\item Identify invariant vector bundles about the torus (ODE and map) as
in section \ref{subsec:Bundles}.
\item Calculate the sought after invariant manifold directly (ODE and map).
(not detailed here)
\item Calculate two invariant foliations (section \ref{subsec:FoilCalc})
and reconstruct the invariant manifold (section \ref{subsec:Reconstruct})
for the map only.
\item \label{enu:frequencies}Recover the frequencies and damping ratios
for the sought after vibration mode (section \ref{subsec:Freq}) for
all models and manifolds identified.
\end{enumerate}
When comparing various methods for calculating the ROM, we use the
relative error 
\begin{equation}
E_{\mathit{rel}}=\frac{\left\Vert \boldsymbol{F}\left(\boldsymbol{W}\left(\boldsymbol{z},\theta\right),\theta\right)-\boldsymbol{W}\left(\boldsymbol{R}\left(\boldsymbol{z},\theta\right),\theta+\omega\right)\right\Vert }{\left\Vert \boldsymbol{W}\left(\boldsymbol{z},\theta\right)\right\Vert }\label{eq:E-rel}
\end{equation}
as a measure of accuracy. We also consider how the error depends on
the amplitude, that is calculated as $A=\left\Vert \boldsymbol{W}\left(\boldsymbol{z},\theta\right)\right\Vert $.

\subsection{\label{subsec:onemass}A planar oscillator}

The first example is a forced geometrically nonlinear oscillator.
This oscillator appeared in \cite{TOUZE200477} without forcing and
damping. The schematic of the oscillator can be seen in figure \ref{fig:OneMass}.
\begin{figure}
\begin{centering}
\includegraphics[scale=0.7]{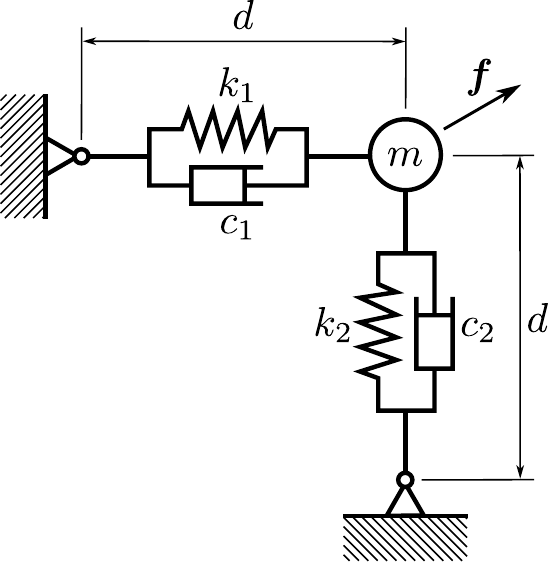}
\par\end{centering}
\caption{\label{fig:OneMass}A planar single mass oscillator in its equilibrium
configuration. Gravity is ignored. The parameters are dimensionless
$m=1$, $d=1$, $k_{1}=1$, $k_{2}$, $c_{1}=0.06$,$c_{2}=0.12$
and the forcing is $\boldsymbol{f}=A\left(\cos\left(\omega t+\pi/3\right),\cos\omega t\right)$,
where $A$ is the forcing amplitude and $\omega$ is the forcing frequency.}
\end{figure}
The equations of motion are
\begin{equation}
\begin{array}{rl}
\dot{x}_{1} & =v_{1}\\
\dot{x}_{2} & =v_{2}\\
\dot{v}_{1} & =-\frac{c_{1}\left(x_{1}+1\right)\left(v_{1}\left(x_{1}+1\right)+v_{2}x_{2}\right)}{\left(x_{1}+1\right){}^{2}+x_{2}^{2}}-\frac{c_{2}x_{1}\left(v_{1}x_{1}+v_{2}\left(x_{2}+1\right)\right)}{x_{1}^{2}+\left(x_{2}+1\right){}^{2}}\\
 & \quad-\frac{k_{1}\left(x_{1}+1\right)\left(\sqrt{\left(x_{1}+1\right){}^{2}+x_{2}^{2}}-1\right)}{\sqrt{\left(x_{1}+1\right){}^{2}+x_{2}^{2}}}-k_{2}x_{1}\left(1-\frac{1}{\sqrt{x_{1}^{2}+\left(x_{2}+1\right){}^{2}}}\right)+A\sin\left(\omega t+\pi/3\right)\\
\dot{v}_{2} & =-\frac{c_{1}x_{2}\left(v_{1}\left(x_{1}+1\right)+v_{2}x_{2}\right)}{\left(x_{1}+1\right){}^{2}+x_{2}^{2}}-\frac{c_{2}\left(x_{2}+1\right)\left(v_{1}x_{1}+v_{2}\left(x_{2}+1\right)\right)}{x_{1}^{2}+\left(x_{2}+1\right){}^{2}}\\
 & \quad-k_{1}x_{2}\left(1-\frac{1}{\sqrt{\left(x_{1}+1\right){}^{2}+x_{2}^{2}}}\right)-\frac{k_{2}\left(x_{2}+1\right)\left(\sqrt{x_{1}^{2}+\left(x_{2}+1\right){}^{2}}-1\right)}{\sqrt{x_{1}^{2}+\left(x_{2}+1\right){}^{2}}}+A\cos\omega t.
\end{array}\label{eq:onemass-ode}
\end{equation}
The unforced natural frequencies are $\omega_{1}=1.0$, $\omega_{2}=1.58$,
the spectral quotients $\beth_{1}=1$ and $\beth_{2}=2$. When forcing
with amplitude $A=0.03$ and frequency $\omega=1.2$, the natural
frequencies become $\omega_{1}=1.031$, $\omega_{2}=1.557$ and the
spectral quotients become $\beth_{1}=1$, \textbf{$\beth_{2}=1.79$}.
\begin{figure}
\begin{centering}
\includegraphics[width=0.99\textwidth]{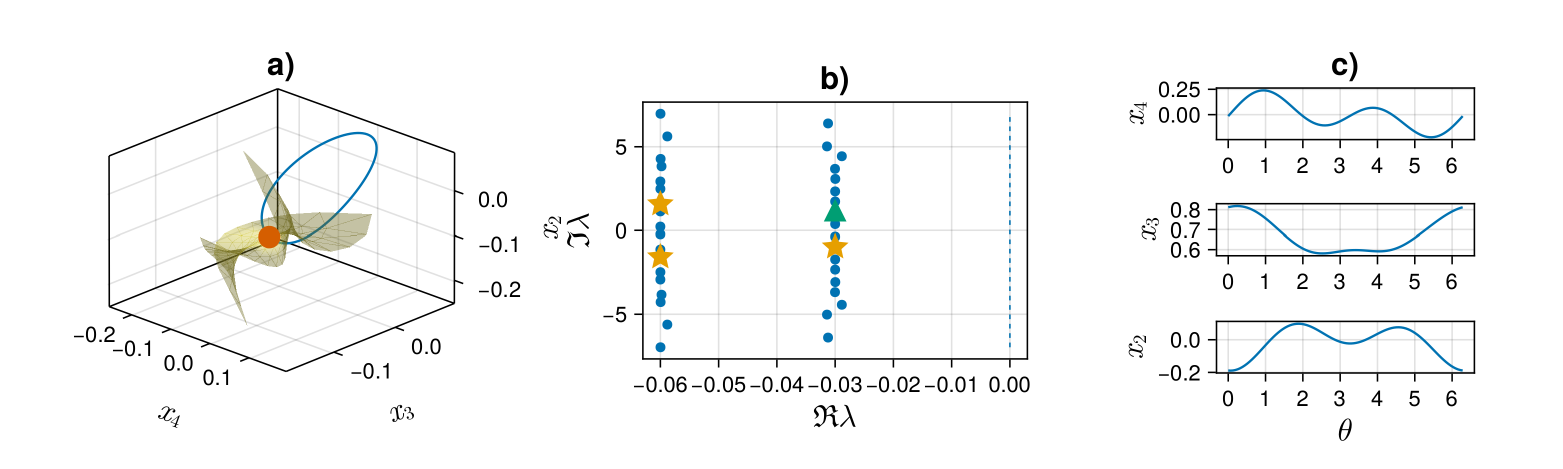}
\par\end{centering}
\caption{\label{fig:onemass-spectrum}a) The invariant torus and the invariant
manifold at one point along the torus. b) The spectrum of the linear
dynamics about the torus. The stars denote the representative eigenvalues
as selected by (\ref{eq:NUM-evsel}), the triangle denotes the eigenvalue
use for model reduction. c) An incomplete representation of the invariant
vector bundle: three out of four coordinates of the first vector that
spans the two-dimensional invariant vector bundle.}

\end{figure}

The spectrum, the invariant torus, the invariant manifold and information
about the vector bundle are depicted in figure \ref{fig:onemass-spectrum}.
The spectrum points in figure \ref{fig:onemass-spectrum}(b) are calculated
as $\lambda=\left(\Delta t\right)^{-1}\log\hat{\lambda}$, where $\hat{\lambda}$
is the spectrum point of the discrete-time dynamics (see section \ref{subsec:Bundles})
and $\Delta t=0.8$ is the sampling period. The four spectrum points
that are used to calculate the vector bundles are denoted by stars
and a triangle. 

We calculate the ROM in three different ways. First we solve the invariance
equation for the ordinary differential equation, which is 
\begin{equation}
D_{1}\boldsymbol{W}\left(\boldsymbol{z},\theta\right)\boldsymbol{R}\left(\boldsymbol{z},\theta\right)+\omega D_{2}\boldsymbol{W}\left(\boldsymbol{z},\theta\right)=\boldsymbol{F}\left(\boldsymbol{W}\left(\boldsymbol{z},\theta\right),\theta\right),\label{eq:MANIF-ODE-invar}
\end{equation}
secondly we solve the invariance equation (\ref{eq:MANIF-MAP-invar})
for the generated Poincare map and finally we use invariant foliations
to obtain the ROM. The result of our calculation can be seen in figure
\ref{fig:OneMass}. 
\begin{figure}
\begin{centering}
\includegraphics[width=0.8\textwidth]{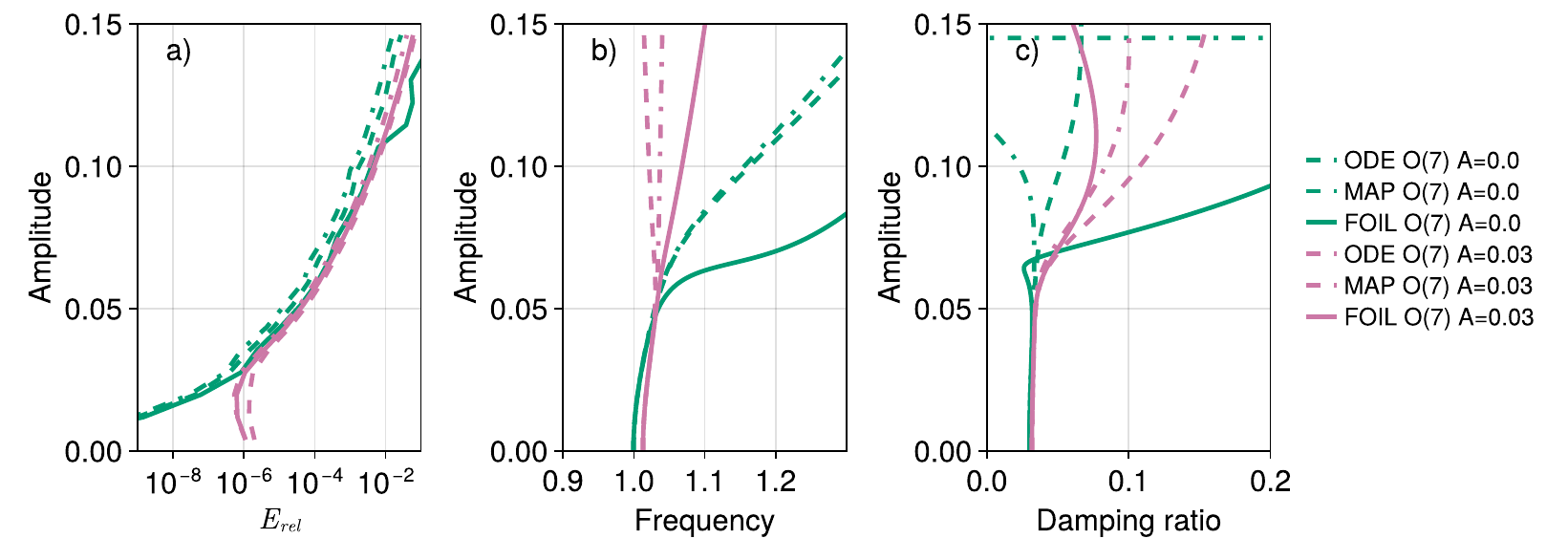}\\
\includegraphics[width=0.8\textwidth]{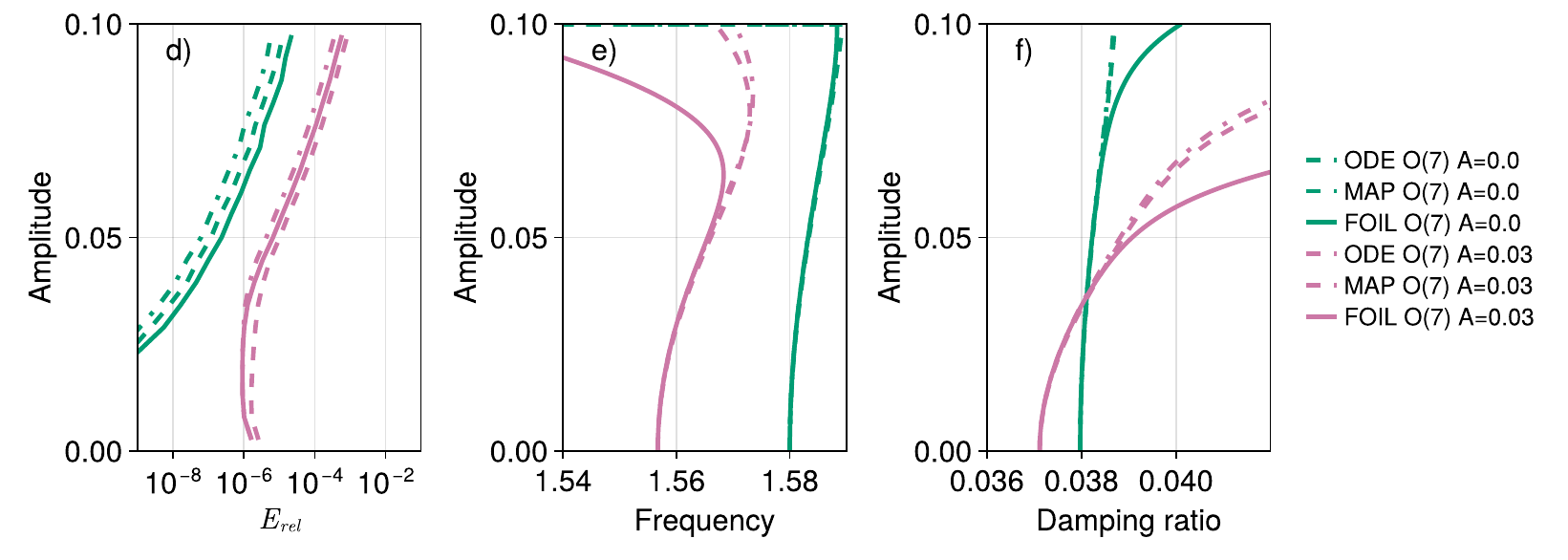}
\par\end{centering}
\caption{\label{fig:OneMass-results}Instantaneous frequency b,e) and damping
ratios c,f) about the invariant torus. The relative errors are shown
in panels a,d). The green lines are calculated for the unforced system,
the purple lines for the forced system with $A=0.03$. The calculations
are carried out directly for the vector field (\ref{eq:onemass-ode})
(ODE), the generated Poincare map with sampling period $\Delta t=0.8$
(MAP) or through the invariant foliations of the Poincare map (FOIL).
All polynomials were of order $7$ and the Fourier collocation used
$\ell=7$ harmonics.}
\end{figure}
 We notice that different calculations only agree for lower amplitudes,
we notice that the disagreement starts to occur when the mean relative
error increases above $10^{-4}$. The error of different calculations
of the same ROM are very close, which cannot explain the different
results. The most likely explanation is that the radius of convergence
of the power series expansion is at lower amplitudes than the maximum
displayed in figure \ref{fig:OneMass-results}. Fitting the invariance
equations to data can overcome the problem with the radius of convergence.
We also note that it was necessary to use order $\ell=7$ Fourier
collocation to reduce the calculation error at zero amplitudes. Order
$\ell=3$ has provided a uniform error just below $10^{-2}$ for all
amplitudes (data not shown).

\subsection{Two-mass oscillator with nonlinear springs and dampers}

As opposed to the geometrically nonlinear system, we consider a system
that is geometrically linear but has nonlinear components. The system,
as shown in figure \ref{fig:TwoMass}, consists of two block masses
sliding on a frictionless surface. The two masses are connected by
nonlinear springs and linear dampers. Forcing is applied to both masses
with a phase shift. 
\begin{figure}
\begin{centering}
\includegraphics[scale=0.7]{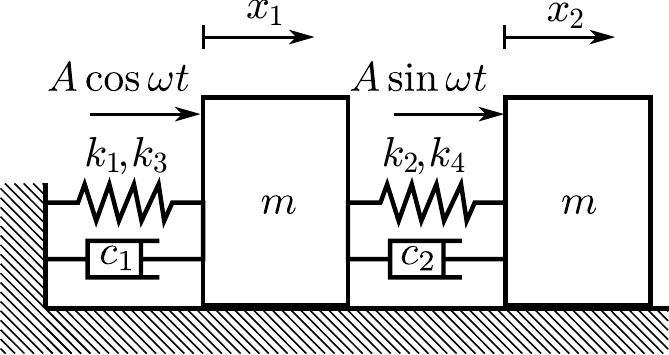}
\par\end{centering}
\caption{\label{fig:TwoMass}A two-mass oscillator. The parameters are dimensionless
$m=1$, $d=1$, $k_{1}=1$, $k_{2}=\sqrt{3}-A/2\sin^{2}\omega t-A/4\sin\omega t$,
$k_{3}=k_{4}=0.02$, $c_{1}=0.03$,$c_{2}=0.04$ and $\omega=0.76$
is the forcing frequency.}
\end{figure}
The equations of motion are
\begin{equation}
\begin{array}{rl}
\dot{x}_{1} & =v_{1}\\
\dot{x}_{2} & =v_{2}\\
\dot{v}_{1} & =-d_{1}v_{1}+d_{2}\left(v_{2}-v_{1}\right)-2k_{4}\left(x_{1}-x_{2}\right){}^{3}-k_{2}\left(x_{1}-x_{2}\right)-2k_{3}x_{1}^{3}-k_{1}x_{1}+A\cos\omega t\\
\dot{v}_{2} & =-d_{2}\left(v_{2}-v_{1}\right)-2k_{4}\left(x_{2}-x_{1}\right){}^{3}-k_{2}\left(x_{2}-x_{1}\right)+A\sin\omega t.
\end{array}\label{eq:twomass-ode}
\end{equation}
The unforced natural frequencies are $\omega_{1}=0.6551$ and $\omega_{2}=2.008$
with damping ratios $\zeta_{1}=0.0095$ and $\zeta_{2}=0.0243$. The
spectral quotients are $\beth_{1}=1.0$ and $\beth_{2}=7.86$. With
forcing at $A=0.1$ the natural frequencies are $\omega_{1}=0.6739$,
$\omega_{2}=2.0112$ with damping ratios $\zeta_{1}=0.0093$ and $\zeta_{2}=0.0242$.
With forcing, the spectral quotients become $\beth_{1}=1.0$ and $\beth_{2}=7.81$.
\begin{figure}
\begin{centering}
\includegraphics[width=0.99\textwidth]{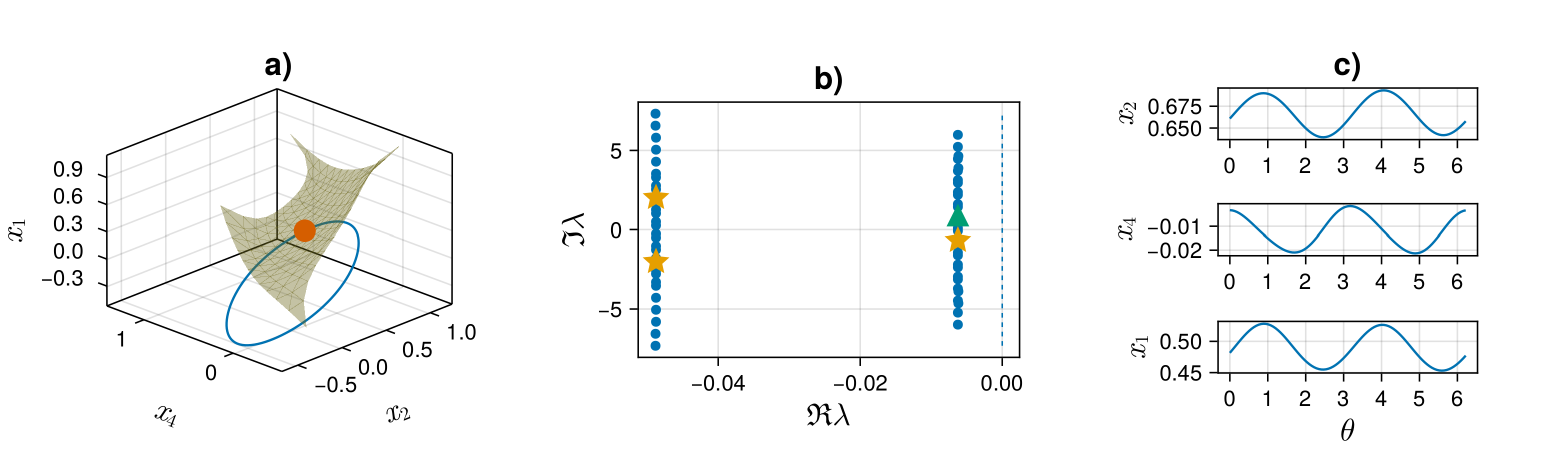}
\par\end{centering}
\caption{\label{fig:twomass-spectrum}a) The invariant torus and the invariant
manifold at one point along the torus. b) The spectrum of the linear
dynamics about the torus. The stars denote the representative eigenvalues
as selected by (\ref{eq:NUM-evsel}), the triangle denotes the eigenvalue
use for model reduction. c) An incomplete representation of the invariant
vector bundle: three out of four coordinates of the first vector that
spans the two-dimensional invariant vector bundle.}
\end{figure}
The results can be seen in figure \ref{fig:TwoMass-results}. When
calculating the ROMs, we notice that we can achieve better accuracy
than in the previous example of section \ref{subsec:onemass}. The
instantaneous frequencies agree well, but there are differences for
the forced system at higher amplitude, which again could be due to
the radius of convergence of the asymptotic power series expansion.
The instantaneous damping shows greater differences. In general, we
find that instantaneous damping ratios are sensitive to how they are
calculated numerically.
\begin{figure}
\begin{centering}
\includegraphics[width=0.8\textwidth]{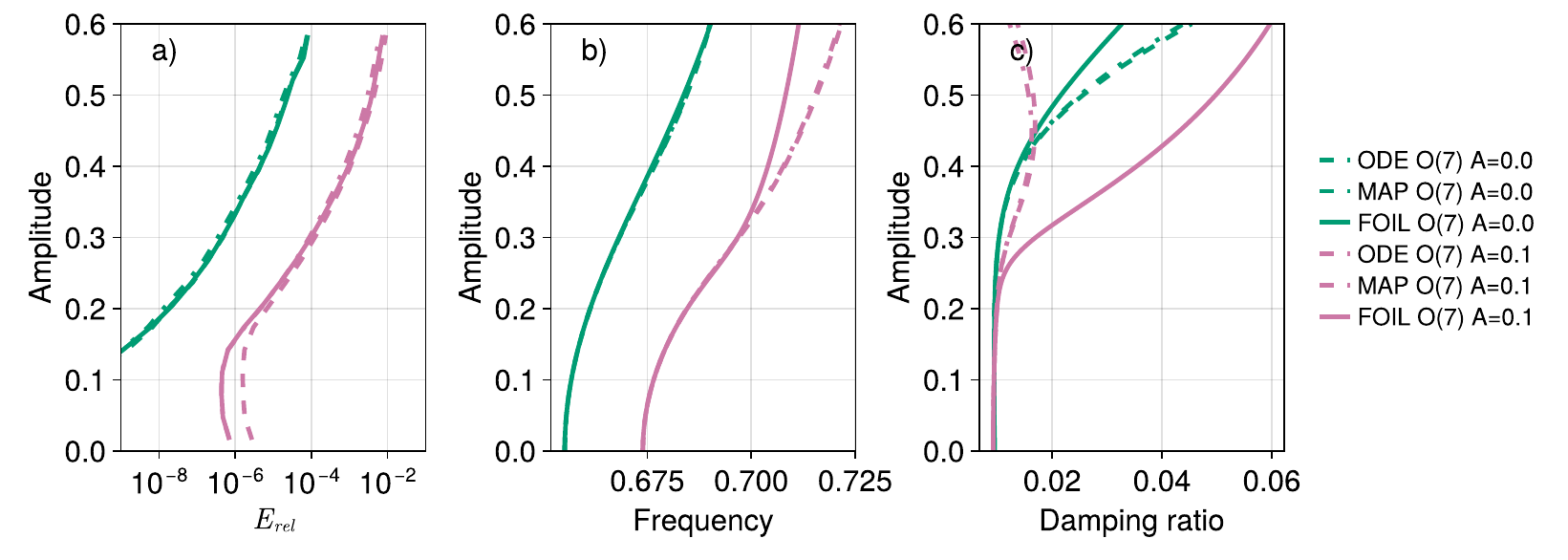}\\
\includegraphics[width=0.8\textwidth]{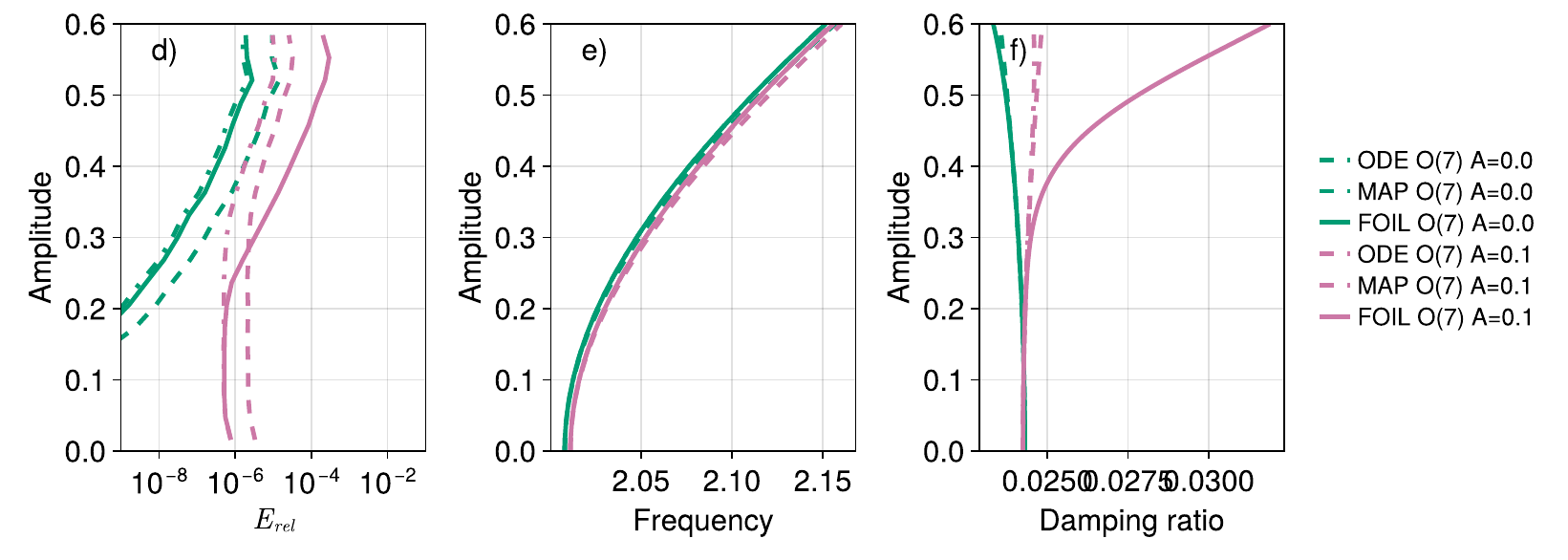}
\par\end{centering}
\caption{\label{fig:TwoMass-results}Instantaneous frequency b),e) and damping
ratios c),f) about the invariant torus. The relative errors are shown
in a),d). The green lines are calculated for the unforced system,
the purple lines for the forced system with $A=0.03$. The calculations
are carried out directly for the vector field (\ref{eq:twomass-ode})
(ODE), the generated Poincare map with sampling period $\Delta t=0.8$
(MAP) or through the invariant foliations of the Poincare map (FOIL).
All polynomials were of order $7$ and the Fourier collocation also
used $\ell=7$ harmonics.}
\end{figure}

As a final calculation, we demonstrate that even if we only calculate
the conjugate map $\boldsymbol{R}$ up to linear order, we can still
recover nonlinear behaviour that is encoded within the coordinate
system represented by the invariant foliations and manifolds. In figure
\ref{fig:twomass-O1} we calculated the invariant manifolds of the
vector field (\ref{eq:twomass-ode}) to order 7, while the conjugate
map $\boldsymbol{R}$ remained linear when the invariant manifold
and the invariant foliations were calculated. Figure \ref{fig:twomass-O1}
shows that while the relative accuracy $E_{\mathit{\ensuremath{rel}}}$
of the calculation reduces with assuming a linear $\boldsymbol{R}$,
for small amplitudes we can reconstruct frequencies and damping rations
accurately. For higher amplitudes the instantaneous frequency and
damping curves diverge earlier than in figure \ref{fig:TwoMass-results}.

\begin{figure}
\begin{centering}
\includegraphics[width=0.8\textwidth]{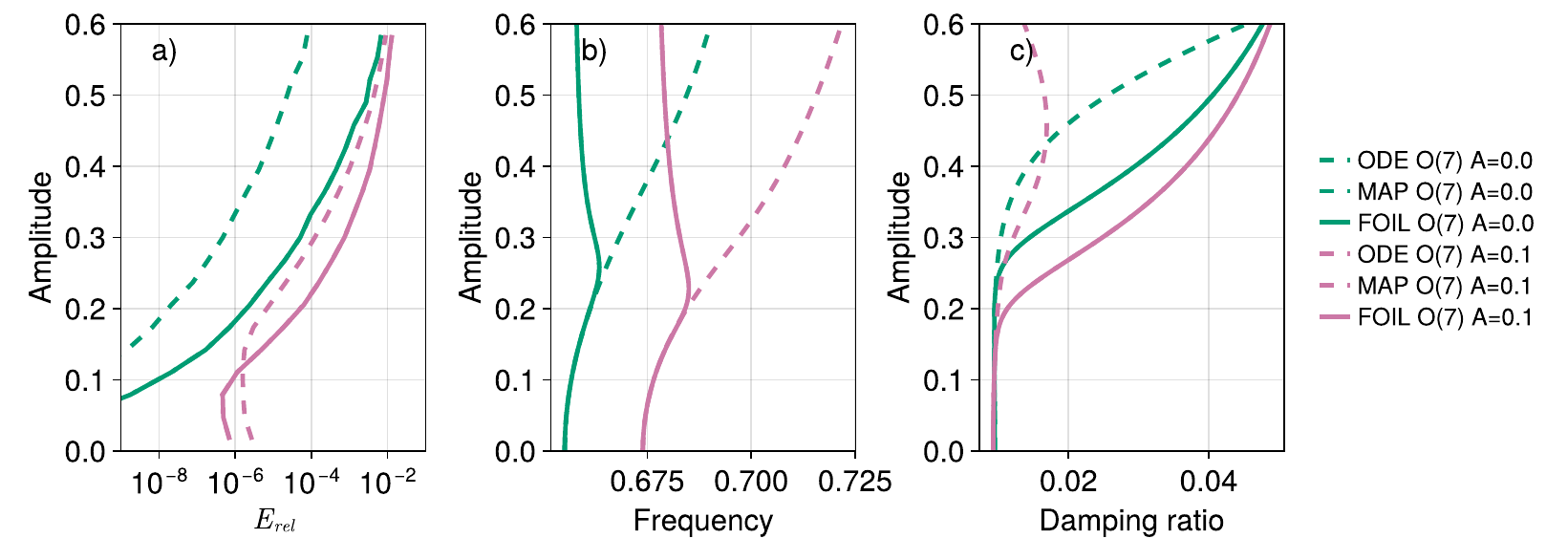}
\par\end{centering}
\caption{\label{fig:twomass-O1}This is a re-calculation of figure \ref{fig:TwoMass-results}(a-c)
assuming that the conjugate map $\boldsymbol{R}$ is linear for the
MAP and FOIL results. Note that the MAP and FOIL results are also
identical in this case.}
\end{figure}

\section{Conclusions}

We have shown the existence and uniqueness of invariant foliations
about quasi-periodic tori. The results are similar to the autonomous
case detailed in \cite{Szalai2020ISF}, but they are more restrictive,
because the non-resonance conditions involve not only individual spectral
points, but concentric circles. In practice, however, the resonance
conditions will remain discrete, even though the discretised transfer
operator still has many more spectral points than in the autonomous
case. In this paper we have eliminated all non-autonomous resonances
and therefore the normal form of the ROM became autonomous. This allowed
us to depict the vibrations about the invariant torus with frequencies
and damping ratios independent of where they are along the torus.
In situations where the forcing has resonances with the underlying
system, this is not possible.

The accuracy of the calculations was also investigated, showcasing
the differences between various methods of calculating ROMs. We conjecture
that the differences are due to the asymptotic nature of the calculations,
which has a radius of convergence that is smaller than the range of
amplitudes the we might be interested in.

In the companion paper \cite{SzalaiForcedData2024}, we investigate
how fitting invariant foliations directly to data affects accuracy.
We will also highlight that in a data-driven setting the smoothness
criterion of theorem \ref{thm:exists-unique} for uniqueness is unhelpful,
because the data fitting is global as opposed to asymptotic about
the invariant torus. Therefore our ongoing work focusses on finding
better uniqueness criteria that better aligns with data-driven methods.

\appendix

\section{Newton-Kantorovich theorem}

The main tool that we use for establishing the existence and uniqueness
of the invariant torus and the invariant foliation is the Newton-Kantorovich
theorem. The theorem is useful as it replaces contraction mapping
arguments used elsewhere.
\begin{thm}
\label{thm:Kantorovich}Let $X$ and $Y$ be Banach spaces and consider
an open subset $\Omega$ of $X$, a point $\boldsymbol{x}_{0}\in\Omega$
and a map $\mathcal{T}\in C^{1}\left(\Omega,Y\right)$ such that both
$D\mathcal{T}\left(\boldsymbol{x}_{0}\right),D\mathcal{T}^{-1}\left(\boldsymbol{x}_{0}\right)\in\mathcal{L}\left(X,Y\right)$.
Assume three constants $\lambda,\mu,\nu$ such that
\begin{gather*}
0<\lambda\mu\nu<\frac{1}{2}\quad\text{and}\quad B_{r}\left(\boldsymbol{x}_{0}\right)\subset\Omega,\quad\text{where}\quad r=\frac{1+\sqrt{1-2\lambda\mu\nu}}{\mu\nu},\\
\left\Vert D\mathcal{T}^{-1}\left(\boldsymbol{x}_{0}\right)\mathcal{T}\left(\boldsymbol{x}_{0}\right)\right\Vert _{X}\le\lambda,\\
\left\Vert D\mathcal{T}^{-1}\left(\boldsymbol{x}_{0}\right)\right\Vert _{\mathcal{L}\left(Y,X\right)}\le\mu,\\
\left\Vert D\mathcal{T}\left(\boldsymbol{x}\right)-D\mathcal{T}\left(\hat{\boldsymbol{x}}\right)\right\Vert _{\mathcal{L}\left(X,Y\right)}\le\nu\left\Vert \boldsymbol{x}-\hat{\boldsymbol{x}}\right\Vert _{X}\quad\text{for all}\quad\boldsymbol{x},\hat{\boldsymbol{x}}\in B_{r}\left(\boldsymbol{x}_{0}\right).
\end{gather*}
Then the sequence $\boldsymbol{x}_{0},\ldots,\boldsymbol{x}_{k+1}=\boldsymbol{x}_{k}-\left[D\mathcal{T}\left(\boldsymbol{x}_{k}\right)\right]^{-1}\boldsymbol{x}_{k}$
is contained in the ball $B_{\rho}\left(\boldsymbol{x}_{0}\right)$
where
\[
\rho=\frac{1-\sqrt{1-2\lambda\mu\nu}}{\mu\nu}\le r
\]
 and converges to the unique zero of $\mathcal{T}$ in $\overline{B_{\rho}\left(\boldsymbol{x}_{0}\right)}$.
\end{thm}
\begin{proof}
This is a subset of theorem 3 in \cite{Ciarlet2012249}. See also
\cite{Fernandez20201} for further versions of the theorem.
\end{proof}

\section{Existence and uniqueness of the invariant torus}

An invariant foliation is always specific to an underlying invariant
object, which in our case is an invariant torus. The torus can be
represented as a graph over the space $\mathbb{T}^{d},$ which we
denote by

The existence of a unique invariant torus can only be guaranteed locally
near an approximate solution of the invariance equation (\ref{eq:TOR-invariances}).
The condition besides having a good approximation is that the approximate
solution is hyperbolic. The following result is well-know, for example
from \cite{HaroTori2006,Haro2016}.
\begin{thm}
Assume that $\boldsymbol{K}_{0}\in C^{a}\left(\mathbb{T}^{d},X\right)$
such that $\boldsymbol{E}\left(\boldsymbol{\theta}\right)=\boldsymbol{F}\left(\boldsymbol{K}_{0}\left(\boldsymbol{\theta}\right),\boldsymbol{\theta}\right)-\boldsymbol{K}_{0}\left(\boldsymbol{\theta}+\omega\right)$
and that $1\in\mathcal{R}\left(\boldsymbol{A}_{0};\boldsymbol{\omega}\right)$,
where $\boldsymbol{A}_{0}\left(\boldsymbol{\theta}\right)=D_{1}\boldsymbol{F}\left(\boldsymbol{K}_{0}\left(\boldsymbol{\theta}\right),\boldsymbol{\theta}\right)$.
Then for a sufficiently small $\boldsymbol{E}\in C^{a}\left(\mathbb{T}^{d},X\right)$,
equation (\ref{eq:TOR-invariances}) has a unique solution for $\boldsymbol{K}$
in $C^{a}\left(\mathbb{T}^{d},X\right)$ in a small neighbourhood
of $\boldsymbol{K}_{0}$.
\end{thm}
\begin{proof}
Since this has been detailed in \cite{HaroTori2006}, we just provide
a sketch of the proof. Define the operator $\mathcal{T}:C^{a}\left(\mathbb{T}^{d},X\right)\to C^{a}\left(\mathbb{T}^{d},X\right)$
as
\begin{equation}
\mathcal{T}\left(\boldsymbol{\Delta}\right)\left(\boldsymbol{\theta}\right)=\boldsymbol{\Delta}\left(\boldsymbol{\theta}+\omega\right)-\boldsymbol{A}\left(\boldsymbol{\theta}\right)\boldsymbol{\Delta}\left(\boldsymbol{\theta}\right)-\boldsymbol{N}\left(\boldsymbol{\Delta}\left(\boldsymbol{\theta}\right),\boldsymbol{\theta}\right)-\boldsymbol{E}\left(\boldsymbol{\theta}\right).\label{eq:TOR-implicit}
\end{equation}
The Frechet derivative of operator $\mathcal{T}$at the origin is
\[
\left(D\mathcal{T}\left(\boldsymbol{\Delta}\right)\boldsymbol{\kappa}\right)\left(\boldsymbol{\theta}\right)=\boldsymbol{\kappa}\left(\boldsymbol{\theta}+\omega\right)-\left[\boldsymbol{A}\left(\boldsymbol{\theta}\right)+D_{1}\boldsymbol{N}\left(\boldsymbol{\Delta}\left(\boldsymbol{\theta}\right),\boldsymbol{\theta}\right)\right]\boldsymbol{\kappa}\left(\boldsymbol{\theta}\right).
\]
Now we refer to theorem \ref{thm:Kantorovich} to show that (\ref{eq:TOR-implicit})
has a unique solution. By our assumption $1\in\mathcal{R}\left(\boldsymbol{A};\boldsymbol{\omega}\right)$,
$D\mathcal{T}\left(\boldsymbol{0}\right)$ is invertible due to corollary
\ref{cor:ED-linear-solution}, hence $\mu$ is finite. Since $D_{1}\boldsymbol{N}$
is linearly increasing with $\boldsymbol{\Delta}$ the Lipschitz constant
$\nu$ is also finite. In addition we can make $\lambda$ as small
as necessary because $\boldsymbol{E}$ is sufficiently small due to
$\boldsymbol{K}_{0}$ being a good approximation. Finally $\left[D\mathcal{T}\left(\boldsymbol{\Delta}\right)\boldsymbol{\kappa}-D\mathcal{T}\left(\hat{\boldsymbol{\Delta}}\right)\boldsymbol{\kappa}\right]\left(\boldsymbol{\theta}\right)=\left[D_{1}\boldsymbol{N}\left(\hat{\boldsymbol{\Delta}}\left(\boldsymbol{\theta}\right),\boldsymbol{\theta}\right)-D_{1}\boldsymbol{N}\left(\boldsymbol{\Delta}\left(\boldsymbol{\theta}\right),\boldsymbol{\theta}\right)\right]\boldsymbol{\kappa}\left(\boldsymbol{\theta}\right)$,
which implies that $\nu\to0$ as $r\to0$ with a linear rate. Therefore,
if $\boldsymbol{E}$ is sufficiently small, we can achieve $\lambda\mu\nu<\frac{1}{2}$,
which implies a unique solution to (\ref{eq:TOR-implicit}).
\end{proof}

\paragraph{Software}

The computer code that produced the results in this paper can be found
at \href{https://github.com/rs1909/InvariantModels}{https://github.com/rs1909/InvariantModels}

\bibliographystyle{plain}
\bibliography{../../../Bibliography/AllRef}

\end{document}